\documentclass[twoside,11pt,reqno]{amsart}
\usepackage{amsmath,amssymb,amscd,mathrsfs,epic,wasysym,latexsym,tikz,mathrsfs,cite,hyperref}
\usepackage{pb-diagram}
\usepackage[matrix,arrow]{xy}

\usepackage{esvect}

\usepackage[enableskew]{youngtab}
\usepackage{ytableau}
\usepackage{color}
\usepackage{enumerate}
\usepackage[margin=1.1in]{geometry}

\usetikzlibrary{shapes,snakes,calendar,matrix,backgrounds,folding}

\usepackage{young}

\makeatletter

\usepackage{tikz}
\usetikzlibrary{decorations.pathreplacing,decorations.pathmorphing,calc}

\numberwithin{equation}{section}

\newtheorem{Lemma}[equation]{Lemma}
\newtheorem{Theorem}[equation]{Theorem}
\newtheorem{Corollary}[equation]{Corollary}
\theoremstyle{definition}  
\newtheorem{Definition}[equation]{Definition}

\newtheorem*{thmA}{Theorem A}
\newtheorem*{thmB}{Theorem B}

\theoremstyle{remark}
\newtheorem{qqq}{Question}

\let\<\langle
\let\>\rangle

\newcommand\Comment[2][\relax]{\space\par\medskip\noindent%
   \fbox{\begin{minipage}{\textwidth}\textbf{Comment\ifx\relax#1\else---#1\fi}\newline%
        #2\end{minipage}}\medskip
}

\def\b1{\text{\boldmath$1$}}

\newcommand{\cg}{\textup{cg}}

\def\phi{{\varphi}}

\newcommand{\TT}{{\mathcal T}}
\newcommand{\DD}{{\mathcal D}}
\newcommand{\EEE}{{\mathcal E}}

\def\b{\mathfrak{b}}
\def\k{\Bbbk}

\def\TT{\mathcal{T}}

{\catcode`\|=\active
  \gdef\set#1{\mathinner{\lbrace\,{\mathcode`\|"8000%
  \let|\midvert #1}\,\rbrace}}
}
\def\midvert{\egroup\mid\bgroup}

\colorlet{darkgreen}{green!50!black}
\tikzset{dots/.style={very thick,loosely dotted},
         greendot/.style={fill,circle,color=darkgreen,inner sep=1.5pt,outer sep=0}
}
\def\greendot(#1,#2){\node[greendot] at(#1,#2){}}

\newenvironment{braid}{
  \begin{tikzpicture}[baseline=6mm,blue,line width=1pt, scale=0.4,
                      draw/.append style={rounded corners},
                      every node/.append style={font=\fontsize{5}{5}\selectfont}]%
  }{\end{tikzpicture}
}

\def\Grid(#1,#2){
  \draw[very thin,gray,step=2mm] (0,0)grid(#1,#2);
  \draw[very thin,darkgreen,step=10mm] (0,0)grid(#1,#2);
}

\newcommand\Tableau[2][\relax]{
  \begin{tikzpicture}[scale=0.5,draw/.append style={thick,black}]
    \ifx\relax#1\relax%
    \else 
      \foreach\box in {#1} { \filldraw[blue!30]\box+(-.5,-.5)rectangle++(.5,.5); }
    \fi
    \newcount\row\newcount\col
    \row=0
    \foreach \Row in {#2} {
       \col=1
       \foreach\k in \Row {
          \draw(\the\col,\the\row)+(-.5,-.5)rectangle++(.5,.5);
          \draw(\the\col,\the\row)node{\k};
          \global\advance\col by 1
       }
       \global\advance\row by -1
    }
  \end{tikzpicture}
}

\newcommand\YoungDiagram[2][\relax]{
  \begin{tikzpicture}[scale=0.5,draw/.append style={thick,black}]
    \ifx\relax#1\relax%
    \else 
    \foreach\box in {#1} {
      \filldraw[blue!30]\box rectangle ++(1,1);
    }
    \fi
    \newcount\row
    \row=0
    \foreach \col in {#2} {
       \draw(1,\the\row)grid ++(\col,1);
       \global\advance\row by -1
    }
  \end{tikzpicture}
}

\begin{document}


\title[Cooperative half-guards in art galleries]{{\bf Cooperative half-guards in art galleries}}

\author{\sc Daniel Florentino}
\address{Washington \& Jefferson College\\ Washington\\ PA~15301, USA}
\email{florentinodc@washjeff.edu}

\author{\sc Ethan Moy}
\address{Washington \& Jefferson College\\ Washington\\ PA~15301, USA}
\email{moye@washjeff.edu}

\author{\sc Robert Muth}
\address{Department of Mathematics\\ Washington \& Jefferson College\\ Washington\\ PA~15301, USA}
\email{rmuth@washjeff.edu}



\begin{abstract}
In any simple polygonal art gallery \(P\) with \(n\) walls, we show that it is possible to place \(\lfloor n/2 \rfloor -1\) guards whose range of vision is \(180^\circ\) in such a way that every interior point of the gallery can be seen by one of them, and such that the mutual visibility graph formed by the guards is connected. This upper bound is tight, in that there exist galleries which require this number of guards, and equals the known result for guards with full \(360^\circ\) range of vision.
We also show that for orthogonal art galleries, this result may be improved to \(\lfloor n/2\rfloor -2\) guards with \(180^\circ\) range of vision.
\end{abstract}

\maketitle

\section{Introduction}

The {\em art gallery problem} encompasses a wide range of problems in geometry and combinatorics. In this problem, one considers a simple closed polygon \(P\) as the floor plan of an art gallery. A `guard' within \(P\) is a stationary point with a fixed range of view, and we say that a guard \(g\) sees a point \(p \in P\) if the line segment \(\overline{gp}\) lies within \(P\) and within \(g\)'s range of view. We will restrict our attention in this paper to full-guards (which see \(360^\circ\) around them), and half-guards (which have a visibility range of \(180^\circ\)).  A set of guards \(G\) is said to monitor \(P\) if every point in \(P\) is seen by some guard in \(G\). As reported by Honsberger \cite{Hons}, Klee posed the following question in 1973:

\begin{qqq}\label{Q1} How many full-guards are sufficient to monitor any art gallery with \(n\) walls?
\end{qqq}

This was answered by Chv\'atal in 1975 \cite{chvatal}, who showed that \(\lfloor n/3 \rfloor\) full-guards are always sufficient and sometimes necessary to monitor an art gallery \(P\) with \(n\) walls. In 1978, Fisk \cite{fisk} provided an elegant proof of this same fact based on 3-colorings of a triangulation graph for \(P\). In 2000, Urrutia \cite{Urrutia} posed the related question: 

\begin{qqq}\label{Q2} How many half-guards are sufficient to monitor any art gallery with \(n\) walls?
\end{qqq}

While there certainly exist art galleries which require more half-guards to monitor than full-guards (as in Figs.\,1 and 2), the somewhat surprising answer to Question~\ref{Q2} is that the bound given by Chv\'atal applies to half-guards as well; in 2000, T\'oth \cite{Toth} showed that \(\lfloor n/3 \rfloor\) half-guards are always sufficient to monitor an art gallery \(P\) with \(n\) walls.
\begin{align*}
\begin{array}{ccccc}
\begin{tikzpicture}[scale=0.8]
\draw[very thick, black, line cap=round, fill=lightgray!10] (0,0)--(1.5,1.5)--(2,3)--(2.5,1.5)--(4,0)--(2,0.5)--(0,0);
%
\pgfmathsetmacro{\ex}{2}
\pgfmathsetmacro{\ey}{1.15}
\draw[thick, fill=red]  (\ex, \ey) circle (0.15);
\fill[fill=red, opacity = 0.4]  (\ex, \ey) circle (0.5);
%
\end{tikzpicture}
&
&
\begin{tikzpicture}[scale=0.8]
\draw[very thick, black, line cap=round, fill=lightgray!10] (0,0)--(1.5,1.5)--(2,3)--(2.5,1.5)--(4,0)--(2,0.5)--(0,0);
%
\pgfmathsetmacro{\ex}{2.5}
\pgfmathsetmacro{\ey}{1.5}
\pgfmathsetmacro{\rotter}{109}
\fill[fill = red, opacity = 0.4] (\ex,\ey) ++(\rotter:.5) arc (\rotter:\rotter + 180:.5);
\draw[thick, fill=red]  (\ex, \ey) circle (0.15);
\pgfmathsetmacro{\ex}{1.5}
\pgfmathsetmacro{\ey}{1.5}
\pgfmathsetmacro{\rotter}{0-109}
\fill[fill = red, opacity = 0.4] (\ex,\ey) ++(\rotter:.5) arc (\rotter:\rotter + 180:.5);
\draw[thick, fill=red]  (\ex, \ey) circle (0.15);
\end{tikzpicture}
\\
\scriptstyle{\textup{{ Fig.\,1: One full-guard monitors \(P\)}}} & & \scriptstyle{\textup{{ Fig.\,2: Two half-guards monitor \(P\)}}}
\end{array}
\end{align*}

A modification of the art gallery problem that has garnered attention is the notion of cooperative guards, proposed by Liaw, Huang and Lee \cite{LHL}. The mutual visibility graph of a guard set \(G\) is the graph whose nodes consist of the guards themselves, with an edge connecting two guards if they see each other. Then a guard set is said to be cooperative if the mutual visibility graph is connected; the idea being that cooperative guards can effectively pass messages to each other via a chain of line-of-sight communication. We have the obvious extension of Question~\ref{Q1} to this setting:

\begin{qqq}\label{Q3}
How many cooperative full-guards are sufficient to monitor any art gallery with \(n\) walls?
\end{qqq}

Hern\'andez-Pe\~nalver \cite{HP} settled this question by showing that \(\lfloor n/2 \rfloor -1\) cooperative full-guards are always sufficient and sometimes necessary to monitor an art gallery \(P\) with \(n \geq 4\) walls.

\begin{align*}
\begin{array}{c}
\begin{tikzpicture}[scale=0.8]
\draw[very thick, black, line cap=round, fill=lightgray!10] (0,0)--(-1,1.5)--(1,3)--(1.5,2)--(2.3,3)--(5.7,3)--(5.35,2.4)--(5.9,1.5)--(6.5,2.7)--(7.5,1.8)--(8,3)--(11,1)--(11.5,1.5)--(11.5,2.2)--(10.5,1.7)-- (9,3)--(12,2.7)--(12.5,1)--(13.5,0.5)--(12,0.5)--(11,0)--(10,0)--(9,1.8)--(8,0)--(4.5,0.5)--(5,1.3)--(3.5,2.2)--(4,1.5)--(2.5,0)--(0,1)--(0,0);
%
%
%
%
%
\draw[very thick, cyan, dotted, line cap=round] (-0.35,2)--(2,1);
\draw[very thick, cyan, dotted, line cap=round] (2.9,2.7)--(2,1);
\draw[very thick, cyan, dotted, line cap=round] (2.9,2.7)--(7,0.75);
\draw[very thick, cyan, dotted, line cap=round] (4.5,0.6)--(7,0.75);
\draw[very thick, cyan, dotted, line cap=round] (9,2.33)--(7,0.75);
\draw[very thick, cyan, dotted, line cap=round] (9,2.23)--(11.5,0.25);
\draw[very thick, cyan, dotted, line cap=round] (10.0,0.7)--(11.5,0.25);
\draw[very thick, cyan, dotted, line cap=round] (12,2.7)--(11.5,0.25);
\pgfmathsetmacro{\ex}{-0.35}
\pgfmathsetmacro{\ey}{2}
\pgfmathsetmacro{\rotter}{217}
\fill[fill = red, opacity = 0.4] (\ex,\ey) ++(\rotter:.5) arc (\rotter:\rotter+180:.5);
\draw[thick, fill=red]  (\ex,\ey) circle (0.15);
\pgfmathsetmacro{\ex}{2}
\pgfmathsetmacro{\ey}{1}
\pgfmathsetmacro{\rotter}{-10}
\fill[fill = red, opacity = 0.4] (\ex,\ey) ++(\rotter:.5) arc (\rotter:\rotter+180:.5);
\draw[thick, fill=red]  (\ex,\ey) circle (0.15);
\pgfmathsetmacro{\ex}{2.9}
\pgfmathsetmacro{\ey}{2.7}
\pgfmathsetmacro{\rotter}{200}
\fill[fill = red, opacity = 0.4] (\ex,\ey) ++(\rotter:.5) arc (\rotter:\rotter+180:.5);
\draw[thick, fill=red]  (\ex,\ey) circle (0.15);
\pgfmathsetmacro{\ex}{4.5}
\pgfmathsetmacro{\ey}{0.5}
\pgfmathsetmacro{\rotter}{-40}
\fill[fill = red, opacity = 0.4] (\ex,\ey) ++(\rotter:.5) arc (\rotter:\rotter+180:.5);
\draw[thick, fill=red]  (\ex,\ey) circle (0.15);
\pgfmathsetmacro{\ex}{7}
\pgfmathsetmacro{\ey}{0.75}
\pgfmathsetmacro{\rotter}{20}
\fill[fill = red, opacity = 0.4] (\ex,\ey) ++(\rotter:.5) arc (\rotter:\rotter+180:.5);
\draw[thick, fill=red]  (\ex,\ey) circle (0.15);
\pgfmathsetmacro{\ex}{9}
\pgfmathsetmacro{\ey}{2.33}
\pgfmathsetmacro{\rotter}{146}
\fill[fill = red, opacity = 0.4] (\ex,\ey) ++(\rotter:.5) arc (\rotter:\rotter+180:.5);
\draw[thick, fill=red]  (\ex,\ey) circle (0.15);
\pgfmathsetmacro{\ex}{11.5}
\pgfmathsetmacro{\ey}{0.25}
\pgfmathsetmacro{\rotter}{50}
\fill[fill = red, opacity = 0.4] (\ex,\ey) ++(\rotter:.5) arc (\rotter:\rotter+180:.5);
\draw[thick, fill=red]  (\ex,\ey) circle (0.15);
\pgfmathsetmacro{\ex}{10.0}
\pgfmathsetmacro{\ey}{0.7}
\pgfmathsetmacro{\rotter}{-90}
\fill[fill = red, opacity = 0.4] (\ex,\ey) ++(\rotter:.5) arc (\rotter:\rotter+180:.5);
\draw[thick, fill=red]  (\ex,\ey) circle (0.15);
\pgfmathsetmacro{\ex}{12}
\pgfmathsetmacro{\ey}{2.7}
\pgfmathsetmacro{\rotter}{150}
\fill[fill = red, opacity = 0.4] (\ex,\ey) ++(\rotter:.5) arc (\rotter:\rotter+180:.5);
\draw[thick, fill=red]  (\ex,\ey) circle (0.15);
%
%
%
%
\end{tikzpicture}
\\
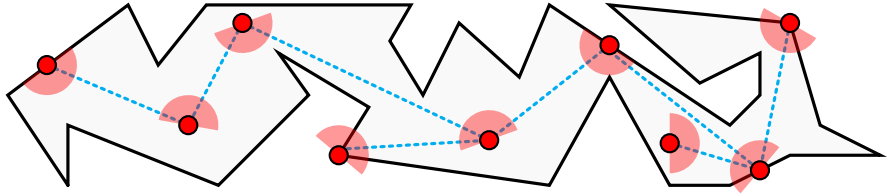
\scriptstyle{\textup{{ Fig.\,3: A cooperative half-guard set, with mutual visibility graph indicated with dashed lines}}}
\end{array}
\end{align*}

To the authors' knowledge, no work has yet been done in studying the notion of cooperative guards with less-than-full viewing range, though this topic was raised in \cite{Zylinski}. For half-guards, there is a wider range of notions of cooperativeness that one could consider, as now one guard may see another but not vice versa. In this paper we still utilize the rather strong criteria that a set of half-guards is cooperative only if the mutual visibility graph is connected, as we make the assumption that two guards can only effectively communicate if they see one another simultaneously. An example of a cooperative half-guard set for an art gallery \(P\) is shown in Fig.\,3 above. At the intersection of Questions~\ref{Q2} and~\ref{Q3} is:

\begin{qqq}\label{Q4}
How many cooperative half-guards are sufficient to monitor any art gallery with \(n\) walls?
\end{qqq}

We propose to answer Question~\ref{Q4} in this paper. As with T\'oth's result in the non-cooperative setting, it turns out that the bound proven in \cite{HP} for cooperative full-guards remains in place for cooperative half-guards. In Theorem~\ref{MainThm} and Corollary~\ref{MainCor}, we show:

\begin{thmA}
{\em 
For \(n \geq 4\), \(\lfloor n/2 \rfloor -1\) cooperative half-guards are always sufficient and sometimes necessary to monitor an art gallery with \(n\) walls. 
}
\end{thmA}

Our result is in fact slightly stronger than stated in Theorem A, as we show that if \(P\) is an even-sided art gallery, then the given cooperative half-guard set may be constructed with a half-guard aligned along any given wall in \(P\). 
Additionally, Hern\'andez-Pe\~nalver \cite{HP} showed that if \(P\) is an {\em orthogonal} art gallery with \(n\) walls---i.e., all walls of \(P\) are at right angles to each other---then the bound for cooperative full-guards may be slightly improved to \(\lfloor n/2 \rfloor -2\). We extend this result to cooperative half-guards as well:

\begin{thmB}
{\em 
For \(n \geq 6\), \(n/2 -2\) cooperative half-guards are always sufficient and sometimes necessary to monitor an orthogonal art gallery with \(n\) walls. 
}
\end{thmB}

This theorem appears as Theorem~\ref{MainThmOrth} and Corollary~\ref{MainCorOrth} in the text. Our methods for proving Theorems A and B are inductive and constructive, in that one can inductively follow the steps in the proofs of Theorems~\ref{MainThm} and~\ref{MainThmOrth} in order to construct a cooperative half-guard set of the appropriate size for any art gallery. Our approach requires a mixture of geometrical and combinatorial tools, and as such cannot guarantee that all half-guards are placed at vertices, unlike in the full-guard setting.

The structure of this paper is as follows. In \S\ref{prelimsec} we discuss geometrical and combinatorial preliminaries, and prove a few necessary lemmas on decompositions of polygons. In \S\ref{guardsec} we set our terminology for half-guards. In \S\ref{TriQuadPent} we prove what amounts to base cases of Theorem A for triangles, quadrilaterals, and pentagons. In \S\ref{mergesec}, we describe some conditions under which cooperative half-guard sets for polygons in a decomposition of \(P\) may be merged into a cooperative half-guard set for \(P\). Finally, in \S\ref{mainsec} and \S\ref{mainsecOrth} we prove the main Theorems A and B.

\section{Preliminaries}\label{prelimsec}
Let \(P\) be a simple closed polygon (henceforth, just {\em polygon}) with \(n \geq 3\) sides. We will occasionally designate a polygon via the notation \(P = [v_1, v_2, \ldots, v_n]\), where \(v_1, \ldots, v_n\) are the vertices of \(P\), and \(\overline{v_1 v_2}, \ldots, \overline{v_{n-1} v_n}, \overline{v_n v_1}\) are the sides of \(P\). If the interior angle at a vertex \(v\) of \(P\) is greater than \(180^\circ\), we say that \(v\) is {\em reflex}, and otherwise we say it is {\em convex}.
A polygon is {\em orthogonal} if the interior angle at each vertex is either \(90^\circ\) or \(270^\circ\). Notably, every orthogonal polygon has an even number of sides.

\subsection{Partitions and decompositions}
In this paper, a {\em partition} \(\DD\) of \(P\) is a set of pairwise non-overlapping polygons (i.e., intersecting only in a finite union of line segments) such that \(P = \bigcup_{D \in \DD} D\). We say that \(\DD\) is moreover a {\em decomposition} of \(P\) if the vertices of every \(D \in \DD\) are vertices of \(P\). 

\subsection{Cuts} A {\em cut} in \(P\) is a line segment \(\overline{vw}\), where \(v,w \in \partial P\) and the relative interior of \(\overline{vw}\) is contained in the interior of \(P\). If \(v,w\) are vertices of \(P\), we say the cut \(\overline{vw}\) is a {\em diagonal}.

Assume \(v\) is a reflex vertex in \(P\), adjacent to a vertex \(u\). Then the ray \(\vv{uv}\) passes into the interior of \(P\) after leaving \(v\), and first intersects \(\partial P\) at some point \(b\). We refer to the cut \(\overline{vb}\) as a {\em \(\vv{uv}\)-cut in \(P\)}.

\subsection{Partitions induced by cuts} A cut \(\overline{vw}\) induces a partition of \( \DD = \{P_1, P_2\}\) of \(P\) into polygons \(P_1, P_2\), with side numbers \(n_1, n_2\) respectively, such that \(\overline{vw} \subset \partial P_1, \partial P_2\), and \(P_1 \cap P_2 = \overline{vw}\). We consider some special situations of partitions induced along cuts:
\begin{align*}
\begin{array}{ccccc}
\begin{tikzpicture}[scale=0.8]
\pgfmathsetmacro{\ex}{0.5}
\pgfmathsetmacro{\ey}{0}
\fill[fill = white, opacity = 0.4] (\ex,\ey) ++(30:.5) arc (30:210:.5);
%
\draw[very thick, black, line cap=round, fill=lightgray!10] (1.5,0)--(1,2)--(0.5,1.5)--(1,3)--(2,2)--(2,1)--(3,3)--(4,0.75)--(4.75,1.75)--(3.5,3)--(5,3)--(5.5,1,0)--(4,0)--(3,1)--(1.5,0);
\draw[very thick, cyan, dotted, line cap=round] (3,1)--(3,3);
\draw[very thick, black, line cap=round] (1.5,0)--(1,2)--(0.5,1.5)--(1,3)--(2,2)--(2,1)--(3,3)--(4,0.75)--(4.75,1.75)--(3.5,3)--(5,3)--(5.5,1,0)--(4,0)--(3,1)--(1.5,0);
%
\node[] at (1.55,1.75) {$P_1$};
\node[] at (4.9,1.2) {$P_2$};
\node[left] at (2.9,3) {$v$};
\node[below] at (3,0.8) {$w$};
\end{tikzpicture}
&
&
\begin{tikzpicture}[scale=0.8]
\pgfmathsetmacro{\ex}{0.5}
\pgfmathsetmacro{\ey}{0}
\fill[fill = white, opacity = 0.4] (\ex,\ey) ++(30:.5) arc (30:210:.5);
%
\draw[very thick, black, line cap=round, fill=lightgray!10] (1.5,0)--(1,2)--(0.5,1.5)--(1,3)--(2,2)--(2,1)--(3,3)--(4,0.75)--(4.75,1.75)--(3.5,3)--(5,3)--(5.5,1,0)--(4,0)--(3,1)--(1.5,0);
\draw[very thick, cyan, dotted, line cap=round] (1.5,0)--(2,1);
\draw[very thick, black, line cap=round] (1.5,0)--(1,2)--(0.5,1.5)--(1,3)--(2,2)--(2,1)--(3,3)--(4,0.75)--(4.75,1.75)--(3.5,3)--(5,3)--(5.5,1,0)--(4,0)--(3,1)--(1.5,0);
%
\node[] at (1.55,1.75) {$P_1$};
\node[] at (4.9,1.2) {$P_2$};
\node[right] at (2.1,1.1) {$v$};
\node[left] at (2.9,3) {$u$};
\node[left] at (1.5,0) {$w$};
\end{tikzpicture}
&
&
\begin{tikzpicture}[scale=0.8]
\pgfmathsetmacro{\ex}{0.5}
\pgfmathsetmacro{\ey}{0}
\fill[fill = white, opacity = 0.4] (\ex,\ey) ++(30:.5) arc (30:210:.5);
%
\draw[very thick, black, line cap=round, fill=lightgray!10] (1.5,0)--(1,2)--(0.5,1.5)--(1,3)--(2,2)--(2,1)--(3,3)--(4,0.75)--(4.75,1.75)--(3.5,3)--(5,3)--(5.5,1,0)--(4,0)--(3,1)--(1.5,0);
\draw[very thick, cyan, dotted, line cap=round] (3,1)--(3.65,1.45);
\draw[very thick, black, line cap=round] (1.5,0)--(1,2)--(0.5,1.5)--(1,3)--(2,2)--(2,1)--(3,3)--(4,0.75)--(4.75,1.75)--(3.5,3)--(5,3)--(5.5,1,0)--(4,0)--(3,1)--(1.5,0);
%
\node[] at (1.55,1.75) {$P_1$};
\node[] at (4.9,1.2) {$P_2$};
\node[right] at (3.65,1.55) {$w$};
\node[below] at (3,0.85) {$v$};
\node[left] at (1.5,0) {$u$};
\end{tikzpicture}
\\
\scriptstyle{\textup{{ Fig.\,4: Diagonal cut}}} & & \scriptstyle{\textup{{ Fig.\,5: Diagonal \(\vv{uv}\)-cut}}} & & \scriptstyle{\textup{{ Fig.\,6: Non-diagonal \(\vv{uv}\)-cut}}}
\end{array}
\end{align*}
\begin{enumerate}
\item Assume \(\overline{vw}\) is a diagonal. Then \(\overline{vw}\) is a side in each of \(P_1, P_2\), \(\mathcal{D}\) is a decomposition of \(P\), and \(n_1 + n_2 = n+2\), as in Fig.\,4.
\item Assume \(\overline{vw}\) is a \(\vv{uv}\)-cut for some reflex vertex \(v\) adjacent to a vertex \(u\). Then \(\overline{uw}\) is a side in one polygon in \(\DD\), and \(\overline{vw}\) is a side in the other polygon in \(\DD\). We have \(n_1 + n_2 = n+1\) if \(w\) is a vertex in \(P\) (so that \(\overline{vw}\) is a diagonal), as in Fig.\,5, and \(n_1 + n_2 = 2\) otherwise, as in Fig.\,6.
\end{enumerate}

\subsection{Trees}\label{Trees}
A {\em tree} is a connected graph with no cycles, or equivalently, a graph in which all nodes are connected by exactly one path. We write \(|G|\) for the number of nodes in \(G\). The {\em degree} of a tree \(G\) is the maximum degree of all nodes in \(G\). A {\em rooted tree} is a tree in which one node \(r\) has been designated a {\em root}. The {\em depth} of a node \(x\) in a rooted tree is the length of the path from \(x\) to the root \(r\). For nodes \(x,y\) in a rooted tree, we say that {\em \(y\) is a descendant of \(x\)} if the path from \(y \) to \(r\) includes the node \(x\). 

A {\em forest} is a graph in which all nodes are connected by at most one path. The connected components of a forest are trees.
If \(G\) is a tree and \(\EEE\) is a subset of edges of \(G\), let \(\hat{G}(\EEE)\) designate the forest obtained by deleting all edges in \(\EEE\) from \(G\). Then \(\hat{G}(\EEE)\) has component trees \(G_0, \ldots, G_{k}\), where \(k = |\EEE|\).

\subsection{Triangulations}\label{Triang}
We say that a decomposition \(\TT\) is a {\em triangulation} of \(P\) if every \(T \in \TT\) is a triangle. It is well known that triangulations always exist, and every triangulation of \(P\) necessarily consists of \(n-2\) triangles. The sides of triangles in \(\TT\) are diagonals or sides in \(P\).

If \(\TT\) is a triangulation of \(P\), we write \(G_\TT\) for the {\em (weak) dual graph} of the triangulation. The nodes of \(G_\TT\) are the triangles in \(\TT\), and two nodes \(T_1,T_2 \in \TT\) are connected by an edge in \(G_\TT\) if and only if the triangles \(T_1,T_2\) share a side; i.e., their intersection is a diagonal in \(P\).
The dual graph \(G_\TT\) is a tree with \(n-2\) nodes and degree less than or equal to three.
\begin{align*}
\begin{array}{ccccc}
\begin{tikzpicture}[scale=0.8]
%
\draw[very thick, black, line cap=round, fill=lightgray!10] (0.5,0)--(0,2)--(1,2)--(2,3)--(3.5,2.5)--(2.5,2)--(3,1)--(4,2)--(3.5,0)--(1.5,0)--(1,1)--(0.5,0);
\draw[very thick, cyan, dotted, line cap=round] (0,2)--(1,1);
\draw[very thick, cyan, dotted, line cap=round] (1,2)--(1,1);
\draw[very thick, cyan, dotted, line cap=round] (3.5,0)--(1,1);
\draw[very thick, cyan, dotted, line cap=round] (3,1)--(1,1);
\draw[very thick, cyan, dotted, line cap=round] (3,1)--(3.5,0);
\draw[very thick, cyan, dotted, line cap=round] (3,1)--(1,2);
\draw[very thick, cyan, dotted, line cap=round] (2.5,2)--(1,2);
\draw[very thick, cyan, dotted, line cap=round] (2.5,2)--(2,3);
\draw[very thick, black, line cap=round] (0.5,0)--(0,2)--(1,2)--(2,3)--(3.5,2.5)--(2.5,2)--(3,1)--(4,2)--(3.5,0)--(1.5,0)--(1,1)--(0.5,0);
%
\end{tikzpicture}
&
&
\begin{tikzpicture}[scale=0.8]
%
\draw[very thick, black, line cap=round, fill=lightgray!10] (0.5,0)--(0,2)--(1,2)--(2,3)--(3.5,2.5)--(2.5,2)--(3,1)--(4,2)--(3.5,0)--(1.5,0)--(1,1)--(0.5,0);
\draw[very thick, cyan, dotted, line cap=round] (0,2)--(1,1);
\draw[very thick, cyan, dotted, line cap=round] (1,2)--(1,1);
\draw[very thick, cyan, dotted, line cap=round] (3.5,0)--(1,1);
\draw[very thick, cyan, dotted, line cap=round] (3,1)--(1,1);
\draw[very thick, cyan, dotted, line cap=round] (3,1)--(3.5,0);
\draw[very thick, cyan, dotted, line cap=round] (3,1)--(1,2);
\draw[very thick, cyan, dotted, line cap=round] (2.5,2)--(1,2);
\draw[very thick, cyan, dotted, line cap=round] (2.5,2)--(2,3);
\draw[very thick, black, line cap=round] (0.5,0)--(0,2)--(1,2)--(2,3)--(3.5,2.5)--(2.5,2)--(3,1)--(4,2)--(3.5,0)--(1.5,0)--(1,1)--(0.5,0);
\draw[thick, fill=black]  (0.55,1) circle (0.1);
\draw[thick, fill=black]  (0.75,1.6) circle (0.1);
\draw[thick, fill=black]  (1.4,1.35) circle (0.1);
\draw[thick, fill=black]  (2.3,1.7) circle (0.1);
\draw[thick, fill=black]  (2.7,0.65) circle (0.1);
\draw[thick, fill=black]  (3.45,1) circle (0.1);
\draw[thick, fill=black]  (1.8,0.35) circle (0.1);
\draw[thick, fill=black]  (1.9,2.35) circle (0.1);
\draw[thick, fill=black]  (2.65,2.45) circle (0.1);
\draw[thick, line cap=round] (0.55,1)--(0.75,1.6);
\draw[thick, line cap=round] (1.4,1.35)--(0.75,1.6);
\draw[thick, line cap=round] (1.4,1.35)--(2.3,1.7);
\draw[thick, line cap=round] (1.4,1.35)-- (2.7,0.65);
\draw[thick, line cap=round] (3.45,1)-- (2.7,0.65);
\draw[thick, line cap=round] (1.8,0.35)-- (2.7,0.65);
\draw[thick, line cap=round]  (2.65,2.45)-- (1.9,2.35);
\draw[thick, line cap=round] (2.3,1.7)-- (1.9,2.35);
\end{tikzpicture}
&
&
\begin{tikzpicture}[scale=0.8]
%
\draw[very thick, white, line cap=round] (0.5,0)--(0,2)--(1,2)--(2,3)--(3.5,2.5)--(2.5,2)--(3,1)--(4,2)--(3.5,0)--(1.5,0)--(1,1)--(0.5,0);
\draw[thick, fill=black]  (0.55,1) circle (0.1);
\draw[thick, fill=black]  (0.75,1.6) circle (0.1);
\draw[thick, fill=black]  (1.4,1.35) circle (0.1);
\draw[thick, fill=black]  (2.3,1.7) circle (0.1);
\draw[thick, fill=black]  (2.7,0.65) circle (0.1);
\draw[thick, fill=black]  (3.45,1) circle (0.1);
\draw[thick, fill=black]  (1.8,0.35) circle (0.1);
\draw[thick, fill=black]  (1.9,2.35) circle (0.1);
\draw[thick, fill=black]  (2.65,2.45) circle (0.1);
\draw[thick, line cap=round] (0.55,1)--(0.75,1.6);
\draw[thick, line cap=round] (1.4,1.35)--(0.75,1.6);
\draw[thick, line cap=round] (1.4,1.35)--(2.3,1.7);
\draw[thick, line cap=round] (1.4,1.35)-- (2.7,0.65);
\draw[thick, line cap=round] (3.45,1)-- (2.7,0.65);
\draw[thick, line cap=round] (1.8,0.35)-- (2.7,0.65);
\draw[thick, line cap=round]  (2.65,2.45)-- (1.9,2.35);
\draw[thick, line cap=round] (2.3,1.7)-- (1.9,2.35);
\end{tikzpicture}
\\
\scriptstyle{\textup{{ Fig.\,7: Triangulation \(\TT\) of \(P\)}}} & & \scriptstyle{\textup{{ Fig.\,8: Superimposed dual graph }}} & & \scriptstyle{\textup{{ Fig.\,9: Dual graph \(G_\TT\)}}}
\end{array}
\end{align*}

If \(\EEE\) is a subset of \(k\) edges in \(G_\TT\), we may specify an associated decomposition \(\hat{\DD}_\TT(\EEE)\) of \(P\) into \(k+1\) polygons, as shown in Figs.\,10--12 below. Let \(\hat{G}_\TT(\EEE)\) be the forest obtained by deleting the edges in \(\EEE\), as described in \S\ref{Trees}, which has component trees \(G_0, \ldots, G_{k}\). For each \(G_i\), we have an associated polygon \(P_i\) which is the union of the nodes (triangles) in \(G_i\). Then \(\hat{D}_\TT(\EEE) = \{P_0, \ldots, P_{k}\}\) is a decomposition of \(P\), where \(P_i, P_j\) intersect in a diagonal \(d\) if and only if \(d\) corresponds to the edge in \(\EEE\) which connects \(G_i\) to \(G_j\) in \(G_\TT\).
Moreover, the nodes of \(G_i\) comprise a triangulation of \(P_i\), so in particular \(P_i\) is a \((|G_i| + 2)\)-sided polygon.

\begin{align*}
\begin{array}{ccccc}
\begin{tikzpicture}[scale=0.8]
%
\draw[very thick, white, line cap=round] (0.5,0)--(0,2)--(1,2)--(2,3)--(3.5,2.5)--(2.5,2)--(3,1)--(4,2)--(3.5,0)--(1.5,0)--(1,1)--(0.5,0);
\draw[thick, line cap=round] (0.55,1)--(0.75,1.6);
\draw[thick, red, dotted, line cap=round] (1.4,1.35)--(0.75,1.6);
\draw[thick,red, dotted,  line cap=round] (1.4,1.35)--(2.3,1.7);
\draw[thick,  line cap=round] (1.4,1.35)-- (2.7,0.65);
\draw[thick, line cap=round] (3.45,1)-- (2.7,0.65);
\draw[thick, red, dotted,  line cap=round] (1.8,0.35)-- (2.7,0.65);
\draw[thick, line cap=round]  (2.65,2.45)-- (1.9,2.35);
\draw[thick, line cap=round] (2.3,1.7)-- (1.9,2.35);
\draw[thick, fill=black]  (0.55,1) circle (0.1);
\draw[thick, fill=black]  (0.75,1.6) circle (0.1);
\draw[thick, fill=black]  (1.4,1.35) circle (0.1);
\draw[thick, fill=black]  (2.3,1.7) circle (0.1);
\draw[thick, fill=black]  (2.7,0.65) circle (0.1);
\draw[thick, fill=black]  (3.45,1) circle (0.1);
\draw[thick, fill=black]  (1.8,0.35) circle (0.1);
\draw[thick, fill=black]  (1.9,2.35) circle (0.1);
\draw[thick, fill=black]  (2.65,2.45) circle (0.1);
\end{tikzpicture}
&
\begin{tikzpicture}[scale=0.8]
%
\draw[very thick, black, line cap=round, fill=lightgray!10] (0.5,0)--(0,2)--(1,2)--(2,3)--(3.5,2.5)--(2.5,2)--(3,1)--(4,2)--(3.5,0)--(1.5,0)--(1,1)--(0.5,0);
\draw[very thick, cyan, dotted, line cap=round] (0,2)--(1,1);
\draw[very thick, cyan, dotted, line cap=round] (1,2)--(1,1);
\draw[very thick, cyan, dotted, line cap=round] (3.5,0)--(1,1);
\draw[very thick, cyan, dotted, line cap=round] (3,1)--(1,1);
\draw[very thick, cyan, dotted, line cap=round] (3,1)--(3.5,0);
\draw[very thick, cyan, dotted, line cap=round] (3,1)--(1,2);
\draw[very thick, cyan, dotted, line cap=round] (2.5,2)--(1,2);
\draw[very thick, cyan, dotted, line cap=round] (2.5,2)--(2,3);
\draw[very thick, black, line cap=round] (0.5,0)--(0,2)--(1,2)--(2,3)--(3.5,2.5)--(2.5,2)--(3,1)--(4,2)--(3.5,0)--(1.5,0)--(1,1)--(0.5,0);
\draw[thick, line cap=round] (0.55,1)--(0.75,1.6);
\draw[thick,  line cap=round] (1.4,1.35)-- (2.7,0.65);
\draw[thick, line cap=round] (3.45,1)-- (2.7,0.65);
\draw[thick, line cap=round]  (2.65,2.45)-- (1.9,2.35);
\draw[thick, line cap=round] (2.3,1.7)-- (1.9,2.35);
\draw[thick, fill=black]  (0.55,1) circle (0.1);
\draw[thick, fill=black]  (0.75,1.6) circle (0.1);
\draw[thick, fill=black]  (1.4,1.35) circle (0.1);
\draw[thick, fill=black]  (2.3,1.7) circle (0.1);
\draw[thick, fill=black]  (2.7,0.65) circle (0.1);
\draw[thick, fill=black]  (3.45,1) circle (0.1);
\draw[thick, fill=black]  (1.8,0.35) circle (0.1);
\draw[thick, fill=black]  (1.9,2.35) circle (0.1);
\draw[thick, fill=black]  (2.65,2.45) circle (0.1);
\end{tikzpicture}
&
\begin{tikzpicture}[scale=0.8]
%
\draw[very thick, black, line cap=round, fill=lightgray!10] (0.5,0)--(0,2)--(1,2)--(2,3)--(3.5,2.5)--(2.5,2)--(3,1)--(4,2)--(3.5,0)--(1.5,0)--(1,1)--(0.5,0);
\draw[very thick, cyan, dotted, line cap=round] (1,2)--(1,1);
\draw[very thick, cyan, dotted, line cap=round] (3.5,0)--(1,1);
\draw[very thick, cyan, dotted, line cap=round] (3,1)--(1,2);
%
\draw[very thick, black, line cap=round] (0.5,0)--(0,2)--(1,2)--(2,3)--(3.5,2.5)--(2.5,2)--(3,1)--(4,2)--(3.5,0)--(1.5,0)--(1,1)--(0.5,0);
%
\end{tikzpicture}
\\
\scriptstyle{\textup{{ Fig.\,10: Selected edges \(\EEE\) in \(G_\TT\)}}} & \scriptstyle{\textup{{ Fig.\,11: Forest \(\hat{G}_\TT(\EEE)\) superimposed on \(P\) }}} & \scriptstyle{\textup{{ Fig.\,12: Decomposition \(\hat{\DD}_\TT(\EEE)\) of \(P\)  }}}
\end{array}
\end{align*}

\subsection{Some useful decomposition results}
In this subsection we establish some decomposition results which will prove useful in \S\ref{mainsec}. In all cases we make use of the method of specifying decompositions using the triangulation dual graph described in \S\ref{Triang}.

\begin{Lemma}\label{OddOdd}
Let \(G\) be a tree, with \(|G|\) even. Let \(x\) be a node in \(G\) of degree one or two. Then there exists an edge \(e\) incident to \(x\) such that \(\hat{G}(\{e\})\) has component trees \(G_0, G_1\) with \(|G_0|, |G_1|\) odd.
\end{Lemma}
\begin{proof}
First assume the degree of \(x\) is one. Taking \(e\) to be the edge incident to \(x\), we have that \(\hat{G}(\{e\})\) has component trees \(G_0 = \{x\}\) and \(G_1\), where \(|G_1| = |G| -1\) is odd. Thus \(e\) satisfies the claim.

Now assume the degree of \(x\) is two. Let \(e_1, e_2\) be the edges incident to \(x\). Then we have that \(\hat{G}(\{e_1, e_2\})\) has component trees \(G_0 = \{x\}, G_1, G_2\), where the edge \(e_1\) connects \(x\) to a node in \(G_1\) and \(e_2\) connects \(x\) to a node in \(G_2\). We have \(|G_1| + |G_2| = |G|-1\), so \(|G_1|, |G_2|\) have opposite parity. Assume without loss of generality that \(|G_1|\) is odd. Then we have that \(\hat{G}(\{e_1\})\) has component trees \(G_1\) and some tree \(G_2'\). As \(|G_1|\) is odd and \(|G_2'| = |G| - |G_1|\) is odd, we have that \(e_1\) satisfies the claim.
\end{proof}

\begin{Corollary}\label{OddOddDiag}
Let \(P\) be a polygon with \(n \geq 4\) sides, where \(n\) is even. Let \(s\) be a side in \(P\). Then there exists a diagonal \(d\) which contains an endpoint of \(s\) and induces a decomposition of \(P\) into a pair of odd-sided polygons.
\end{Corollary}
\begin{proof}
Let \(\TT\) be a triangulation of \(P\). We have that \(|G_\TT| = n -2\) is odd. There exists one triangle \(T\) such that \(s\) is a side in \(T\). Then at most two of the sides of \(T\) are diagonals in \(P\), so the degree of the node \(T\) in \(G_\TT\) is one or two. Then by Lemma~\ref{OddOdd} there exists an edge \(e\) incident to the node \(T\) in \(G_\TT\) such that \(\hat{G}_\TT(\{e\})\) has component trees \(G_0, G_1\) with \(|G_0|, |G_1|\) odd. The edge \(e\) corresponds to some side \(d\) in \(T\) which is a diagonal in \(P\), and thus necessarily contains one of the endpoints of \(s\). Then \(d\) induces a decomposition \(\hat{\DD}(\{e\})\) of \(P\) into odd-sided polygons \(P_1\), \(P_2\), as required.
\end{proof}

\begin{Lemma}\label{QuadPentGraph}
Let \(G\) be a tree with degree less than or equal to three, and \(|G| \geq 3\). Then there exists an edge \(e \in G\) such that \(\hat{G}(\{e\})\) has component trees \(G_0, G_1\), where \(|G_0| \in \{2,3\}\).\end{Lemma}
\begin{proof}
The graph \(G\) has some leaf \(r\). Choose \(r\) to be the root of \(G\). First assume \(G\) has no nodes of degree three. Then let \(y\) be a node of maximum depth in \(G\). As \(|G| \geq 3\), the depth of \(y\) must be two or more, so \(y\) has degree one and has a parent \(x\) of degree two. Let \(e\) be the edge connecting \(x\) to its parent. Then \(\hat{G}(\{e\})\) has component trees \(G_0, G_1\), where \(G_0\) consists of the nodes \(x,y\), satisfying the claim.

Now assume that \(G\) has degree three.
Let $x$ be a node of degree three and maximum depth, so that all descendants of $x$ are less than degree three. Then there exist exactly two leaves $y_1$ and $y_2$ that are descendants of $x$. We consider the two possibilities separately:

\textit{Case 1:}
Assume one of the leaves, say $y_1$, has a degree two parent $z$. Let $e$ be the edge connecting $z$ to its parent. 
Then \(\hat{G}(\{e\})\) has component trees \(G_0, G_1\), where \(G_0\) consists of the nodes \(x,y_1\), satisfying the claim.

\textit{Case 2:}
Assume $y_1$ and $y_2$ both have degree three parents. Then, by the assumption on \(x\), they must have the same parent \(x\).
Let $e$ be the edge connecting $x$ to its parent. 
Then \(\hat{G}(\{e\})\) has component trees \(G_0, G_1\), where \(G_0\) consists of the nodes \(x,y_1,y_2\), satisfying the claim.
\end{proof}

\begin{Corollary}\label{QuadPent}
Let \(P\) be a polygon with \(n \geq 5\) sides. Then there exists a diagonal \(d\) in \(P\) which induces a decomposition of \(P\) into polygons \(P_0, P_1\), where \(P_0\) is a quadrilateral or a pentagon.
\end{Corollary}
\begin{proof}
Let \(\TT\) be a triangulation of \(P\). Then \(|G_\TT| = n-2 \geq 3\), so by Lemma~\ref{QuadPentGraph}, \(G_\TT\) contains some edge \(e\) such that \(\hat{G}(\{e\})\) has component trees \(G_1, G_2\), where \(|G_1| \in \{2,3\}\). Then \(e\) corresponds to a diagonal \(d\) in \(P\) which induces a decomposition \(\hat{D}_\TT(\{e\})\) into polygons \(P_0, P_1\), where \(P_i\) has sides numbering \(|G_i| + 2\). But then \(P_0\) is a quadrilateral or a pentagon, completing the proof.
\end{proof}

\begin{Lemma}\label{PentDecompGraph}
Let \(G\) be a tree of degree less than or equal to three with \(|G| \geq 3\) and \(|G|\) odd. Then there exists \(0 \leq k \leq 5\) and a set of \(k\) edges \(\EEE\) of \(G\) such that the component trees \(G_0, \ldots, G_k\) of \(\hat{G}(\EEE)\) satisfy:
\begin{enumerate}
\item \(|G_0|=3\) and \(|G_1|, \ldots, |G_k|\) are even;
\item For \(1 \leq i \leq k\), there exists an edge \(e_i\) in \(\EEE\) which connects \(G_i\) to \(G_0\).
\end{enumerate}
\end{Lemma}
\begin{proof}
We go by induction on \(|G|\). For the base case, note that if \(|G|=3\), then taking \(k=0\), \(G_0\) satisfies the claim. 

Now assume \(|G| > 3\), and the claim holds for smaller trees by induction. By Lemma~\ref{QuadPentGraph}, there exists an edge \(e \in G\) such that deleting \(e\) from \(G\) leaves component trees \(G'_1, G'_2\), where \(|G'_1| \in \{2,3\}\). First assume \(|G'_0|=3\). Then \(|G'_1| = |G| - 3\) is even, so taking \(k = 1\), \(\EEE = \{e\}\), \(G_0 =G_0'\), \(G_1 = G_1'\) satisfies the claim.

Now assume \(|G'_0|=2\). Then \(|G'_1| = |G| -2\) is odd. By the induction assumption there exists \(0 \leq k'' \leq 5\) and a set of \(k''\) edges \(\EEE''\) of \(G_1'\) with component trees \(G_0'', \ldots, G_{k''}''\) of \(\hat{G}'_1(\EEE'')\) that satisfy \(|G_0''|=3\), \(|G_1''|, \ldots, |G_k''|\) are even, and for \(1 \leq i \leq k''\), there exists an edge \(e_i''\) in \(\EEE\) which connects \(G_i''\) to \(G_0''\). The edge \(e\) in \(G\) must connect \(G_0'\) to exactly one of the components \(G_i''\).

Assume first that \(e\) connects \(G_0'\) to \(G_i''\), for some \(i \in \{1, \ldots, k''\}\). Then, let \(k = k''\), \(\EEE = \EEE''\), \(G_j = G_j''\) for \(0 \leq j \leq k''\) with \(j \neq i\), and let \(G_i\) be the tree consisting of \(G_0'\) and \(G_i''\) connected by the edge \(e\). As \(|G_i| = |G_i''| + 2\) is even, this choice of \(k, \EEE, G_0, \ldots, G_k\) satisfies the claim.

Finally, assume \(e\) connects \(G_0'\) to \(G_0''\). Note that, since the maximum degree of nodes in \(G\) is three, the nodes in \(G_0''\) can be adjacent to at most five nodes outside of \(G_0''\) in \(G\). Therefore we have \(0 \leq k'' \leq 4\). Thus, taking \(k = k'' +1\), \(\EEE = \EEE'' \cup \{e\}\), \(G_i = G_i''\) for \(1 \leq i \leq k''\), and \(G_k = G_0'\) satisfies the claim. This completes the induction step, and the proof.
\end{proof}

\begin{Corollary}\label{PentDecomp}
Let \(P\) be a polygon with \(n \geq 5\) sides, where \(n\) is odd. Then there exists \(0 \leq k \leq 5\) and a decomposition \(\mathcal{D} = \{P_0, P_1, \ldots, P_k\}\) of \(P\) such that 
\begin{enumerate}
\item \(P_0\) is a pentagon, and \(P_1, \ldots, P_k\) are even-sided polygons;
\item For \(1 \leq i \leq k\), \(P_i \cap P_0\) is a diagonal in \(P\).
\end{enumerate}
\end{Corollary}
\begin{proof}
Let \(\TT\) be a triangulation of \(P\). Then \(|G_\TT| = n -2 \geq 3\), so there exists \(0 \leq k \leq 5\) and a set of \(k\) edges \(\EEE\) in \(G_\TT\) which satisfy the conditions of Lemma~\ref{PentDecompGraph}. But then the associated decomposition \(\hat{\DD}_\TT(\EEE)\) satisfies the conditions of the lemma claim, since \(P_0\) has \(|G_0| + 2 = 5\) sides, and \(P_i\) has an even number, \(|G_i| + 2\) sides, for \(1 \leq i \leq k\).
\end{proof}

Note that, as \(P\) is simple, the conditions of Lemma~\ref{PentDecomp} imply that for \(1 \leq i, j \leq k\), \(i \neq j\), the polygons \(P_i, P_j\) may meet at most in a vertex in \(P\). Thus Lemma~\ref{PentDecomp} represents a a decomposition of odd-sided polygons into a pentagonal `hub', and even-sided polygons attached along sides of the pentagon, as in Fig.\,13.

\begin{align*}
\begin{array}{cc}
\begin{tikzpicture}[scale=0.8]
\draw[very thick, black, line cap=round, fill=lightgray!10] (3,0)--(1,4)--(3,7)--(3,3)--(6,3)--(5,1)--(8,2)--(5,6)--(13,7)--(8,5)--(12.5,4.5)--(12,6)--(15,3)--(10.5,3)--(10,2)--(13,1)--(10,0)--(3,0);
\fill[ line cap=round, fill=lightgray!50!] (5,1)--(8,2)--(8,5)--(10,2)--(13,1)--(5,1);
\draw[very thick, black, line cap=round] (3,0)--(1,4)--(3,7)--(3,3)--(6,3)--(5,1)--(8,2)--(5,6)--(13,7)--(8,5)--(12.5,4.5)--(12,6)--(15,3)--(10.5,3)--(10,2)--(13,1)--(10,0)--(3,0);
%
%
%
%
\draw[very thick, cyan, dotted, line cap=round] (5,1)--(13,1);
\draw[very thick, cyan, dotted, line cap=round] (8,2)--(8,5);
\draw[very thick, cyan, dotted, line cap=round] (10,2)--(8,5);
%
%
%
%
\node[below] at (8.75,2.8) {$P_0$};
\node[above] at (9.7,3.4) {$P_1$};
\node[above] at (6.9,4.6) {$P_3$};
\node[] at (3,2) {$P_2$};
\end{tikzpicture}
\\
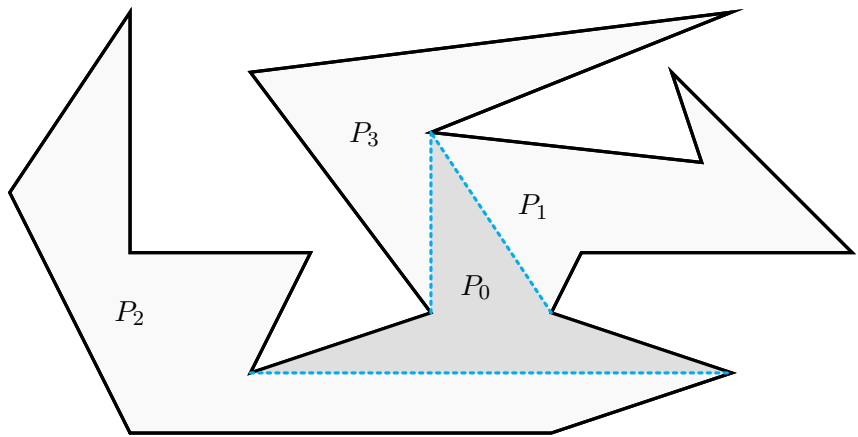
\scriptstyle{\textup{{ Fig.\,13: Decomposition of \(P\) into pentagonal hub \(P_0\) and even-sided polygons \(P_1, P_2, P_3\)}}}
\end{array}
\end{align*}

\section{Guards}\label{guardsec}
A {\em half-guard} in a polygon \(P\) is a pair \({\tt g} = (x_{\tt g},H_{\tt g})\), where \(x_{\tt g}\) is a point in \(P\), and \(H_{\tt g}\) is a closed half-plane such that \(x_{\tt g}\) lies on the boundary \(\partial H_{\tt g}\) of \(H_{\tt g}\). We say \({\tt g}\) is a {\em boundary half-guard for \(P\)} if \(x_{\tt g} \in \partial P\).
If \(y\) is another point in \(P\), we say that \({\tt g}\) {\em sees} \(y\) if the line segment \(\overline{x_{\tt g}y}\) is contained in \(P \cap H_{\tt g}\). 
If \({\tt G}\) is a collection of half-guards, we say that \({\tt G}\) {\em monitors} \(P\) provided that, for all \(y \in P\), there exists \({\tt g} \in {\tt G}\) such that \({\tt g}\) sees \(y\).

We will depict half-guards \({\tt g}\) in diagrams by placing a dot at \(x_{\tt g}\), and indicate the associated half-plane \(H_{\tt g}\) with a semicircle centered at \(x_{\tt g}\) contained in \(H_{\tt g}\), as shown in Fig.\,14.
\begin{align*}
\begin{array}{c}
\begin{tikzpicture}[scale=0.8]
\draw[very thick, black, line cap=round, fill=gray!50] (0,0)--(-1,1.5)--(1,3)--(1.5,2)--(2.3,3)--(5.7,3)--(5.35,2.4)--(5.9,1.5)--(6.5,2.7)--(7.5,1.8)--(8,3)--(11,1)--(11.5,1.5)--(11.5,2.2)--(10.5,1.7)-- (9,3)--(12,2.7)--(12.5,1)--(13.5,0.5)--(12,0.5)--(11,0)--(10,0)--(9,1.8)--(8,0)--(4.5,0.5)--(5,1.3)--(3.5,2.2)--(4,1.5)--(2.5,0)--(0,1)--(0,0);
\fill[line cap=round, fill=lightgray!10] (0,0)--(-1,1.5)--(1,3)--(1.5,2)--(3.7,2)--(4,1.5)--(2.5,0)--(0,1)--(0,0);
\fill[line cap=round, fill=lightgray!10] (2.2,2.8)--(2.3,3)--(3.5,3)--(5.9,1.46)--(6.5,2.7)--(7.525,1.74)--(8.15,2.95)--(9.4,2.05)--(9,1.84)--(8.8,1.45)--(5.85,0.33)--(4.5,0.5)--(5,1.3)--(3.5,2.23)--(2.2,2.8);
\fill[line cap=round, fill=lightgray!10] (9,3)--(12,2.7)--(12.65,0.5)--(12,0.5)--(11,0)--(10.9,0)--(11.5,1.5)--(11.5,2.2)--(10.5,1.7)--(9,3);
%
%
%
%
\draw[very thick, white, dotted, line cap=round] (-0.35,2)--(3.7,2);
\draw[very thick, white, dotted, line cap=round] (7,0.75)--(2.2,2.8);
\draw[very thick, white, dotted, line cap=round] (7,0.75)--(3.5,3);
\draw[very thick, white, dotted, line cap=round] (7,0.75)--(9.4,2.05);
\draw[very thick, white, dotted, line cap=round] (7,0.75)--(8.15,2.95);
\draw[very thick, white, dotted, line cap=round] (12,2.7)--(12.65,0.5);
\draw[very thick, white, dotted, line cap=round] (12,2.7)--(10.9,0);
%
%
%
%
\pgfmathsetmacro{\ex}{7}
\pgfmathsetmacro{\ey}{0.75}
\pgfmathsetmacro{\rotter}{21}
\fill[fill = red, opacity = 0.4] (\ex,\ey) ++(\rotter:.5) arc (\rotter:\rotter+180:.5);
\draw[very thick, white, dotted, line cap=round] (5.85,0.33)--(8.8,1.45);
\draw[thick, fill=red]  (\ex,\ey) circle (0.15);
\draw[very thick, black, line cap=round] (0,0)--(-1,1.5)--(1,3)--(1.5,2)--(2.3,3)--(5.7,3)--(5.35,2.4)--(5.9,1.5)--(6.5,2.7)--(7.5,1.8)--(8,3)--(11,1)--(11.5,1.5)--(11.5,2.2)--(10.5,1.7)-- (9,3)--(12,2.7)--(12.5,1)--(13.5,0.5)--(12,0.5)--(11,0)--(10,0)--(9,1.8)--(8,0)--(4.5,0.5)--(5,1.3)--(3.5,2.2)--(4,1.5)--(2.5,0)--(0,1)--(0,0);
\pgfmathsetmacro{\ex}{-0.35}
\pgfmathsetmacro{\ey}{2}
\pgfmathsetmacro{\rotter}{217}
\fill[fill = red, opacity = 0.4] (\ex,\ey) ++(\rotter:.5) arc (\rotter:\rotter+180:.5);
\draw[thick, fill=red]  (\ex,\ey) circle (0.15);
\pgfmathsetmacro{\ex}{12}
\pgfmathsetmacro{\ey}{2.7}
\pgfmathsetmacro{\rotter}{150}
\fill[fill = red, opacity = 0.4] (\ex,\ey) ++(\rotter:.5) arc (\rotter:\rotter+180:.5);
\draw[thick, fill=red]  (\ex,\ey) circle (0.15);
%
%
%
%
\end{tikzpicture}
\\
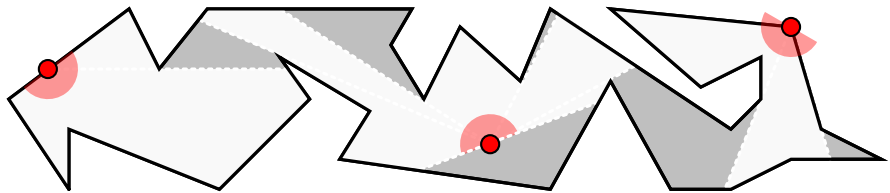
\scriptstyle{\textup{{ Fig.\,14: Half-guards in a polygon \(P\). The regions of \(P\) seen by each guard are illuminated.}}}
\end{array}
\end{align*}

\subsection{Cooperative half-guards}
 If \({\tt g}, {\tt h}\) are half-guards, we will abuse notation slightly and write \({\tt g}\) {\em sees} \({\tt h}\) to mean that \({\tt g}\) sees \(x_{{\tt h}}\). 
For a collection \({\tt G}\) of half-guards, the {\em mutual visibility graph} \(V_{\tt G}\) is the graph with nodes \({\tt G}\), with an edge between nodes \({\tt g}\) and \({\tt h}\) if and only if \({\tt g}\) and \({\tt h}\) see each other.

\begin{Definition}
Let \({\tt G}\) be a collection of half-guards in a polygon \(P\). We say that \({\tt G}\) is a {\em  cooperative half-guard (CHG) set for \(P\)} provided that:
\begin{enumerate}
\item \({\tt G}\) monitors \(P\).
\item \(V_{\tt G}\) is connected.
\end{enumerate}
In other words, every point in \(P\) must be seen by some \({\tt g} \in {\tt G}\), and for every pair \({\tt g}, {\tt h}\) of half-guards in \({\tt G}\), there exists a sequence of half-guards \({\tt g} = {\tt g}_0,\, {\tt g}_1, \ldots, {\tt g}_k = {\tt h}\) in \({\tt G}\) such that \({\tt g}_{i-1}\) and \({\tt g}_{i}\) see each other, for all \(1 \leq i \leq k\).
\end{Definition}

\begin{Definition}\label{schgdef}\(\)
\begin{enumerate}
\item For a polygon \(P\), define \(\textup{cg}(P, 180^\circ)\) to be the minimum number of half-guards required to form a cooperative half-guard set for \(P\), i.e.:
\begin{align*}
\textup{cg}(P, 180^\circ) := \min\{ |{\tt G}| \mid \textup{ \({\tt G}\) is a CHG set for \(P\)}\}
\end{align*}
\item For \(n \geq 3\), define \(\textup{cg}(n, 180^\circ)\) to be the maximum number of half-guards required to form a cooperative half-guard set for any polygon with \(n\) sides, i.e.:
\begin{align*}
\textup{cg}(n, 180^\circ) := \max\{ \textup{cg}(P, 180^\circ) \mid \textup{\(P\) is an \(n\)-sided polygon}\}
\end{align*}
\item For even \(n \geq 4\), define \(\textup{cg}^\perp(n, 180^\circ)\) to be the maximum number of half-guards required to form a cooperative half-guard set for any orthogonal polygon with \(n\) sides, i.e.:
\begin{align*}
\textup{cg}^\perp(n, 180^\circ) := \max\{ \textup{cg}(P, 180^\circ) \mid \textup{\(P\) is an \(n\)-sided orthogonal polygon}\}
\end{align*}
\end{enumerate}
\end{Definition}

For arbitrary \(P\), it is a difficult computational problem (NP-hard, as shown in \cite{LHL}) to work out the value of \(\textup{cg}(P, 180^\circ)\). Our goal in this paper is to give an explicit description of the values \(\textup{cg}(n,180^\circ)\) and \(\textup{cg}^\perp(n,180^\circ)\) for all applicable \(n\).

\subsection{Full-guards}
One can similarly define cooperative {\em full-guard} (CFG) sets for polygons, where full guards are assumed to see points in any direction, not just within an associated half-plane. We use the notation \(\textup{cg}(n, 360^\circ)\), etc., to describe values analogous to Definition~\ref{schgdef} for CFG sets. The following theorem was proved by Hern\'andez-Pe\~nalver \cite{HP}. 

\begin{Theorem}\label{HPthm}\(\)
\begin{enumerate}
\item[\textup{(i)}]
For all \(n \geq 4\) we have
\(
\textup{cg}(n,360^\circ) = \lfloor n/2 \rfloor -1
\).
\item[\textup{(ii)}]
For all even \(n \geq 6\) we have
\(
\textup{cg}^\perp(n,360^\circ) = n/2 -2.
\)
\end{enumerate}
\end{Theorem}

It is clear from definitions that we must have \(\textup{cg}(P, 180^\circ) \geq \textup{cg}(P, 360^\circ)\) and \(\textup{cg}^\perp(P, 180^\circ) \geq \textup{cg}^\perp(P, 360^\circ)\), since if \({\tt G}\) is a CHG set for \(P\), then replacing every half-guard in \({\tt G}\) with a full-guard yields an CFG set for \(P\). These are not equalities in general; Figs.\,1 and 2 show an example of a polygon \(P\) such that \(\textup{cg}(P,360^\circ) = 1\) and \(\textup{cg}(P,180^\circ) = 2\).
While it is not true that \(\textup{cg}(P,180^\circ) = \textup{cg}(P,360^\circ)\) in general, we will demonstrate in Corollaries~\ref{MainCor} and~\ref{MainCorOrth} that we nonetheless have \(\textup{cg}(n,180^\circ) = \textup{cg}(n,360^\circ)\) and \(\textup{cg}^\perp(n,180^\circ) = \textup{cg}^\perp(n,360^\circ)\) for all applicable \(n\).

\subsection{Special half-guards}
We now introduce some terminology used in this paper for special classes of half-guards.

\begin{Definition}
Let \({\tt g}\) be a boundary half-guard for \(P\). Let \(u,v \neq x_{\tt g}\) be vertices of \(P\) adjacent to \(x_{\tt g}\). Let \(C\) be the cone formed by the interior angle \(\angle ux_{\tt g}v\) in \(P\). We say that \({\tt g}\) is an {\em entire} boundary half-guard for \(P\) if \(C \subseteq H_{\tt g}\).
\end{Definition}

In other words, \({\tt g}\) is an entire boundary half-guard for \(P\) if and only if \(x_{\tt g} \in \partial P\) is not a reflex vertex and \(H_{\tt g}\) is oriented in such a way that maximizes the region in \(P\) seen by \({\tt g}\). Some examples of entire and non-entire boundary half-guards are shown in Figs.\,15 and 16.
\begin{align*}
\begin{array}{ccccc}
\begin{tikzpicture}[scale=0.8]
\pgfmathsetmacro{\ex}{0.5}
\pgfmathsetmacro{\ey}{0}
\fill[fill = white, opacity = 0.4] (\ex,\ey) ++(30:.5) arc (30:210:.5);
%
\draw[very thick, black, line cap=round, fill=lightgray!10] (-1,0)--(2,3)--(1.5,1)--(3,0)--(-1,0);
%
\pgfmathsetmacro{\ex}{0.5}
\pgfmathsetmacro{\ey}{0}
\fill[fill = red, opacity = 0.4] (\ex,\ey) ++(0:.5) arc (0:180:.5);
\draw[thick, fill=red]  (\ex, \ey) circle (0.15);
\pgfmathsetmacro{\ex}{2}
\pgfmathsetmacro{\ey}{3}
\fill[fill = red, opacity = 0.4] (\ex,\ey) ++(160:.5) arc (160:340:.5);
\draw[thick, fill=red]  (\ex, \ey) circle (0.15);
\end{tikzpicture}
&
&
\begin{tikzpicture}[scale=0.8]
\draw[very thick, black, line cap=round, fill=lightgray!10] (-1,0)--(2,3)--(1.5,1)--(3,0)--(-1,0);
%
\pgfmathsetmacro{\ex}{0.5}
\pgfmathsetmacro{\ey}{0}
\fill[fill = red, opacity = 0.4] (\ex,\ey) ++(30:.5) arc (30:210:.5);
\draw[thick, fill=red]  (\ex, \ey) circle (0.15);
\pgfmathsetmacro{\ex}{1.5}
\pgfmathsetmacro{\ey}{1}
\fill[fill = red, opacity = 0.4] (\ex,\ey) ++(100:.5) arc (100:280:.5);
\draw[thick, fill=red]  (\ex, \ey) circle (0.15);
\end{tikzpicture}
\\
\scriptstyle{\textup{{ Fig.\,15: Entire boundary half-guards}}} & & \scriptstyle{\textup{{ Fig.\,16: Non-entire boundary half-guards}}}
\end{array}
\end{align*}

The following lemma is clear:
\begin{Lemma}\label{BigToSmall}
Let \({\tt g}\) be an entire boundary half-guard in a polygon \(P\). Let \(Q \subseteq P\) be a polygon which contains \(x_{\tt g}\). Then \({\tt g}\) is an entire boundary half-guard in \(Q\).
\end{Lemma}

\begin{Definition}
Let \(t \subset \partial P\) be a line segment. Let \({\tt g}\) be a half-guard in \(P\) such that \({\tt g}\) is colinear to \(t\), sees every point in \(t\), and \(\partial H_{\tt g}\) is colinear to \(t\) if \(x_{\tt g} \in t\). Then we say that \({\tt g}\) is {\em \(t\)-aligned}. If \({\tt G}\) is a CHG set for \(P\) that contains a \(t\)-aligned half-guard \({\tt g}\), we say that \({\tt G}\) is {\em \(t\)-aligned}.
\end{Definition}

We will use \(t\)-aligned half-guards as a tool for merging CHG sets for polygons in partitions of \(P\) into a CHG set for \(P\). Some examples of \(t\)-aligned and non-\(t\)-aligned half-guards are shown in Figs.\,17 and 18.
\begin{align*}
\begin{array}{ccccc}
\begin{tikzpicture}[scale=0.8]
\pgfmathsetmacro{\ex}{1.35}
\pgfmathsetmacro{\ey}{0.35}
\fill[fill = white, opacity = 0.4] (\ex,\ey) ++(120:.5) arc (120:300:.5);
\draw[very thick, black, line cap=round, fill=lightgray!10] (-1,0)--(2,3)--(1.5,1)--(3,0)--(-1,0);
\draw[very thick, cyan, dotted, line cap=round] (1.5,1)--(1.25,0);
\draw[line width=3pt, line cap=round, black] (2,3)--(1.5,1);
\pgfmathsetmacro{\ex}{2}
\pgfmathsetmacro{\ey}{3}
\fill[fill = red, opacity = 0.4] (\ex,\ey) ++(75:.5) arc (75:255:.5);
\draw[thick, fill=red]  (\ex, \ey) circle (0.15);
\pgfmathsetmacro{\ex}{1.35}
\pgfmathsetmacro{\ey}{0.35}
\fill[fill = red, opacity = 0.4] (\ex,\ey) ++(20:.5) arc (20:200:.5);
\draw[thick, fill=red]  (\ex, \ey) circle (0.15);
\pgfmathsetmacro{\ex}{1.68}
\pgfmathsetmacro{\ey}{1.75}
\fill[fill = red, opacity = 0.4] (\ex,\ey) ++(75:.5) arc (75:255:.5);
\draw[thick, fill=red]  (\ex, \ey) circle (0.15);
\node[] at (2.2,2.2) {$t$};
\end{tikzpicture}
&
&
\begin{tikzpicture}[scale=0.8]
\draw[very thick, black, line cap=round, fill=lightgray!10] (-1,0)--(2,3)--(1.5,1)--(3,0)--(-1,0);
\draw[very thick, cyan, dotted, line cap=round] (1.5,1)--(1.25,0);
\draw[line width=3pt, line cap=round, black] (2,3)--(1.5,1);
\pgfmathsetmacro{\ex}{2}
\pgfmathsetmacro{\ey}{3}
\fill[fill = red, opacity = 0.4] (\ex,\ey) ++(160:.5) arc (160:340:.5);
\draw[thick, fill=red]  (\ex, \ey) circle (0.15);
\pgfmathsetmacro{\ex}{1.35}
\pgfmathsetmacro{\ey}{0.35}
\fill[fill = red, opacity = 0.4] (\ex,\ey) ++(120:.5) arc (120:300:.5);
\draw[thick, fill=red]  (\ex, \ey) circle (0.15);
\pgfmathsetmacro{\ex}{1}
\pgfmathsetmacro{\ey}{2}
\fill[fill = red, opacity = 0.4] (\ex,\ey) ++(225:.5) arc (225:405:.5);
\draw[thick, fill=red]  (\ex, \ey) circle (0.15);
%
\node[] at (2.2,2.2) {$t$};
\end{tikzpicture}
\\
\scriptstyle{\textup{{ Fig.\,17: \(t\)-aligned half-guards}}} & & \scriptstyle{\textup{{ Fig.\,18: Non-\(t\)-aligned half-guards}}}
\end{array}
\end{align*}

\section{Triangles, quadrilaterals and pentagons}\label{TriQuadPent}

In this section we show that polygons with \(n=3,4,5\) sides can be monitored by one half-guard which satisfies certain additional conditions. The lemmas in this section will play a role in constructing CHG sets for polygons of arbitrary size.

\begin{Lemma}\label{Tri}
Let \(P\) be a convex polygon and \({\tt g}\) be an entire boundary half-guard in \(P\). Then \(\{{\tt g}\}\) monitors \(P\). In particular, a triangle is monitored by any entire boundary half-guard.
\end{Lemma}
\begin{proof}
As \(P\) is convex, for each side \(s\) of \(P\), there exists a closed half-plane \(H_s\) with \(\partial H_s\) colinear to \(s\) such that \(P = \bigcap_{s} H_s\). If \(x_{\tt g}\) is on the relative interior of a side \(s\), then \(H_s = H_{\tt g}\), and if \(x_{\tt g}\) is at a vertex where sides \(s,t\) meet, then we have \(H_s \cap H_t \subset H_{\tt g}\) by the definition of entire boundary half-guards. Thus \(P \subseteq H_{\tt g}\). Moreover, by convexity of \(P\), we have \(\overline{x_{\tt g} p} \subset P\) for all \(p \in P\), so the result follows.
\end{proof}

\begin{Lemma}\label{MiniQuad}
Let \(Q\) be a quadrilateral, \(v\) be a convex vertex of \(Q\), and assume that both vertices adjacent to \(v\) are convex as well. If \({\tt g}\) is an entire boundary half-guard at \(v\), then \(\{{\tt g}\}\) monitors \(Q\).
\end{Lemma}
\begin{proof}
If all vertices of \(Q\) are convex, then \(Q\) is convex and thus the statement follows by Lemma~\ref{Tri}. Thus we may assume that \(Q = [u,v,w,x]\), where \(x\) is a reflex vertex in \(Q\). There exists a triangulation of \(Q\) of size two, and since \(x\) is reflex, it must belong to both triangles in the triangulation. Thus \(\TT = \{[v,w,x], [v,u,x]\}\). Note then that \({\tt g}\) is an entire boundary half-guard in both triangles in \(\TT\) by Lemma~\ref{BigToSmall}, so \(\{{\tt g}\}\) monitors each by Lemma~\ref{Tri}. Therefore \(\{{\tt g}\}\) monitors \(Q\).
\end{proof}

\begin{Lemma}\label{Quad}
Let \(Q\) be a quadrilateral and \(s\) be any side in \(Q\). There exists an \(s\)-aligned entire boundary half-guard \({\tt g}\) such that \(\{{ \tt g}\}\) monitors \(Q\).
\end{Lemma}
\begin{proof}
We describe the guard \({\tt g}\) in three different cases. First, assume \(Q\) is convex, as in Fig.\,19. Then we may place an \(s\)-aligned entire boundary half-guard \({\tt g}\) at any point in the relative interior of \(s\). Then \(\{{\tt g}\}\) will be \(s\)-aligned and monitor \(Q\) by Lemma~\ref{Tri}.
\begin{align*}
\begin{array}{ccccc}
\begin{tikzpicture}[scale=0.8]
\draw[very thick, black, line cap=round, fill=lightgray!10] (0,0)--(2,3)--(4,1)--(3,0)--(0,0);
\draw[line width=3pt, black, line cap=round] (3,0)--(0,0);
\pgfmathsetmacro{\ex}{1.9}
\pgfmathsetmacro{\ey}{0}
\fill[fill = red, opacity = 0.4] (\ex,\ey) ++(0:.5) arc (0:180:.5);
\draw[thick, fill=red]  (1.9,0) circle (0.15);
%
\node[] at (1.9,-0.5) {$b$};
\node[] at (0.8,-0.4) {$s$};
\end{tikzpicture}
&
&
\begin{tikzpicture}[scale=0.8]
\draw[very thick, black, line cap=round, fill=lightgray!10] (0,0)--(2,3)--(1.5,1)--(3,0)--(0,0);
\draw[line width=3pt, line cap=round, black] (2,3)--(1.5,1);
\draw[very thick, cyan, dotted, line cap=round] (1.5,1)--(1.25,0);
\pgfmathsetmacro{\ex}{1.25}
\pgfmathsetmacro{\ey}{0}
\fill[fill = red, opacity = 0.4] (\ex,\ey) ++(0:.5) arc (0:180:.5);
\draw[thick, fill=red]  (1.25,0) circle (0.15);
%
\node[] at (1.25,-0.5) {$b$};
\node[] at (2.1,2) {$s$};
\end{tikzpicture}
&
&
\begin{tikzpicture}[scale=0.8]
\draw[very thick, black,line cap=round,  fill=lightgray!10] (0,0)--(2,3)--(1.5,1)--(3,0)--(0,0);
\draw[line width=2.5pt,line cap=round,  black] (3,0)--(0,0);
\draw[very thick, cyan, dotted, line cap=round] (1.5,1)--(0,0);
\pgfmathsetmacro{\ex}{0}
\pgfmathsetmacro{\ey}{0}
\fill[fill = red, opacity = 0.4] (\ex,\ey) ++(0:.5) arc (0:180:.5);
\draw[thick, fill=red]  (0,0) circle (0.15);
%
\node[] at (1.25,-0.4) {$s$};
\node[] at (0,-0.5) {$b$};
\end{tikzpicture}
\\
\scriptstyle{\textup{{ Fig.\,19: \(Q\) convex}}} & & \scriptstyle{\textup{{ Fig.\,20: Reflex vertex on \(s\)}}}
& & \scriptstyle{\textup{{ Fig.\,21: Reflex vertex not on \(s\)}}}
\end{array}
\end{align*}

Next, assume one of the endpoints of \(s\) is a reflex vertex, as in Fig.\,20. Then \(s = \overline{wv}\) for some vertices \(v,w\) in \(Q\), with \(v\) reflex. The ray \(\vv{wv}\) passes into the interior of \(Q\) and intersects \(\partial Q\) at some point \(b\), which cannot be a reflex vertex, as \(Q\) can have at most one reflex vertex. We place an entire boundary half-guard \({\tt g}\) at the point \(b\). The quadrilateral \(Q\) may be partitioned along the \(\vv{wv}\)-cut \(\overline{wb}\) into two triangles \(T_1\), \(T_2\), and since \({\tt g}\) is an entire boundary half-guard in each triangle by Lemma~\ref{BigToSmall}, we have by Lemma~\ref{Tri} that \(\{{\tt g}\}\) monitors \(Q = T_1 \cup T_2\). Moreover, as \(x_{\tt g} \notin s\) and \({\tt g}\) sees \(s\), we have that \({\tt g}\) is \(s\)-aligned.

Finally, assume that \(Q\) is non-convex, but both endpoints of \(s\) are convex vertices, as in Fig.\,21. Let \(b\) be the vertex of \(s\) nonadjacent to a reflex angle of \(Q\). Let \({\tt g}\) be the entire boundary half-guard for \(P\) at \(b\) with half-plane oriented such that \(\partial H_{\tt g}\) is colinear to \(s\). Then \({\tt g}\) is \(s\)-aligned, and it follows by Lemma~\ref{MiniQuad} that \(\{{\tt g}\}\) monitors \(Q\).
\end{proof}

\begin{Lemma}\label{Pent}
Let \(P\) be a pentagon. There exists an entire boundary half-guard \({\tt g}\) for \(P\) such that
\(\{ {\tt g}\}\) monitors \(P\), and \(x_{\tt g}\) is not a vertex of \(P\).
\end{Lemma}
\begin{proof}
The pentagon \(P\) has \(0,1\) or \(2\) reflex vertices. If \(P\) has no reflex vertices, then it is convex, and so an entire boundary half-guard placed along the relative interior of any side of \(P\) satisfies the claim by Lemma~\ref{Tri}.

Now assume \(P\) has one reflex vertex. Let \(P = [u,v,w,x,y]\), with \(v\) reflex. The ray \(\vv{uv}\) passes from \(v\) into the interior of \(P\) and intersects \(\partial P\) at some point \(b\). If \(b\) is not a vertex, as in Fig.\,22, then we may place an entire boundary half-guard \({\tt g}\) for \(P\) at \(b\). Then \(P\) may be partitioned along the \(\vv{uv}\)-cut \(\overline{vb}\) into two convex polygons \(P_1\) and \(P_2\). As \({\tt g}\) is an entire boundary half-guard in each of the convex polygons \(P_1\), \(P_2\) by Lemma~\ref{BigToSmall}, we have that \(\{{\tt g}\}\) monitors \(P = P_1 \cup P_2\) by Lemma~\ref{Tri}.
\begin{align*}
\begin{array}{cccc}
\begin{tikzpicture}[scale=0.8]
\draw[very thick, black, line cap=round, fill=lightgray!10] (1.2,1.5)--(2,1)--(2.5,3)--(3.5,0)--(0.8,0)--(1.2,1.5);
\draw[very thick, cyan, dotted, line cap=round] (1.75,0)--(2,1);
\pgfmathsetmacro{\ex}{1.75}
\pgfmathsetmacro{\ey}{0}
\fill[fill = red, opacity = 0.4] (\ex,\ey) ++(0:.5) arc (0:180:.5);
\draw[thick, fill=red]  (1.75,0) circle (0.15);
%
\node[right] at (3.5,0) {$y$};
\node[left] at (1.2,1.5) {$w$};
\node[above] at (1.85,1.15) {$v$};
\node[right] at (2.5,3) {$u$};
\node[left] at (0.8,0) {$x$};
\node[] at (1.75,-0.5) {$b$};
\end{tikzpicture}
&&
\begin{tikzpicture}[scale=0.8]
\draw[very thick, black, line cap=round, fill=lightgray!10] (0.5,2)--(2,1)--(2.5,3)--(3.5,0)--(1.75,0)--(0.5,2);
\draw[very thick, cyan, dotted, line cap=round] (1.75,0)--(2,1);
\draw[very thick, cyan, dotted, line cap=round] (3.5,0)--(2,1);
\draw[very thick, cyan, dotted, line cap=round] (2.4,0)--(2,1);
\pgfmathsetmacro{\ex}{2.4}
\pgfmathsetmacro{\ey}{0}
\fill[fill = red, opacity = 0.4] (\ex,\ey) ++(0:.5) arc (0:180:.5);
\draw[thick, fill=red]  (2.4,0) circle (0.15);
%
\node[right] at (3.5,0) {$y$};
\node[left] at (0.5,2) {$w$};
\node[above] at (1.85,1.15) {$v$};
\node[right] at (2.5,3) {$u$};
\node[left] at (1.72,0) {$x$};
\node[] at (2.4,-0.5) {$b$};
\end{tikzpicture}
\\
\scriptstyle{\textup{{ Fig.\,22: Half-guard on ray from reflex vertex}}} & & \scriptstyle{\textup{{ Fig.\,23: Half-guard on reflex angle bisector}}}
\end{array}
\end{align*}
Similarly, if the ray \(\vv{wv}\) intersects \(\partial P\) at a non-vertex point \(b\), placing an entire boundary half-guard at \(b\) satisfies the claim.
Assume then that \(\vv{uv}\) intersects \(\partial P\) at \(x\) and \(\vv{wv}\) intersects \(\partial P\) at \(y\), as in Fig.\,23. Then there exists some \(b\) in the relative interior of \(\overline{xy}\) such that \(\overline{vb}\) is a bisector for the interior angle at \(v\). Then \(P\) may be partitioned along \(\overline{vb}\) into two convex polygons \(P_1\) and \(P_2\). Place an entire boundary half-guard \({\tt g}\) for \(P\) at \(b\). We have then that \({\tt g}\) is an entire boundary half-guard for \(P_1\) and \(P_2\) by Lemma~\ref{BigToSmall}, and so \(\{{\tt g}\}\) monitors \(P = P_1 \cup P_2\) by Lemma~\ref{Tri}.

Finally, assume that \(P\) has two reflex vertices. Then there are two possible cases; either the reflex vertices are adjacent or nonadjacent. First assume that the reflex vertices are nonadjacent. Let \(P = [u,v,w,x,y]\) with \(u,w\) reflex, as in Fig.\,24. Then the ray \(\vv{vw}\) must intersect \(\partial P\) at some point \(b\) in the relative interior of the side \(\overline{xy}\). Then \(P\) may be partitioned along the \(\vv{vw}\)-cut \(\overline{wb}\) into the triangle \(T=[b,w,x]\) and the quadrilateral \(Q=[b,y,u,v]\). Place an entire boundary half-guard \({\tt g}\) for \(P\) at \(b\). We have then that \({\tt g}\) is an entire boundary half-guard for \(P_1\) and \(P_2\) by Lemma~\ref{BigToSmall}, and thus \(\{{\tt g}\}\) monitors \(P = P_1 \cup P_2\) by Lemmas~\ref{Tri} and \ref{MiniQuad}.
\begin{align*}
\begin{array}{cccc}
\begin{tikzpicture}[scale=0.8]
\draw[very thick, black, line cap=round, fill=lightgray!10] (0.5,0)--(-0.5,0.5)--(-2,3)--(-1.5,1)--(-2.5,0)--(0.5,0);
%
%
\draw[very thick, cyan, dotted, line cap=round] (-1.5,1)--(-1.25,0);
\pgfmathsetmacro{\ex}{-1.25}
\pgfmathsetmacro{\ey}{0}
\fill[fill = red, opacity = 0.4] (\ex,\ey) ++(0:.5) arc (0:180:.5);
\draw[thick, fill=red]  (-1.25,0) circle (0.15);
\node[] at (-1.25,-0.5) {$b$};
\node[right] at (-0.5,0.8) {$u$};
\node[left] at (-2,3) {$v$};
\node[left] at (-1.6,1.2) {$w$};
\node[below] at (-2.5,0) {$x$};
\node[below] at (0.5,0) {$y$};
\end{tikzpicture}
&&
\begin{tikzpicture}[scale=0.8]
\draw[very thick, black, line cap=round, fill=lightgray!10] (-3,0)--(0,0)--(-2,3)--(-1.75,2)--(-2,0.5)--(-3,0);
\draw[very thick, cyan, dotted, line cap=round] (-1.75,2)--(-1.25,0);
\pgfmathsetmacro{\ex}{-1.25}
\pgfmathsetmacro{\ey}{0}
\fill[fill = red, opacity = 0.4] (\ex,\ey) ++(0:.5) arc (0:180:.5);
\draw[thick, fill=red]  (-1.25,0) circle (0.15);
%
\node[] at (-1.25,-0.5) {$b$};
\node[left] at (-2,3) {$u$};
\node[left] at (-1.75,2) {$v$};
\node[left] at (-2,0.7) {$w$};
\node[below] at (-3,0) {$x$};
\node[below] at (0,0) {$y$};
\end{tikzpicture}
\\
\scriptstyle{\textup{{ Fig.\,24: Nonadjacent reflex vertices}}} & & \scriptstyle{\textup{{ Fig.\,25: Adjacent reflex vertices}}}
\end{array}
\end{align*}
Next, assume that the reflex vertices are adjacent. Let \(P = [u,v,w,x,y]\), with \(v,w\) reflex, as in Fig.\,25. Then the ray \(\vv{uv}\) must intersect \(\partial P\) at some point \(b\) in the relative interior of the side \(\overline{xy}\). Then \(P\) may be partitioned along the \(\vv{uv}\)-cut \(\overline{vb}\) into the triangle \(T=[b,y,u]\) and the quadrilateral \(Q=[b,v,w,x]\). Place an entire boundary half-guard \({\tt g}\) for \(P\) at \(b\). We have then that \({\tt g}\) is an entire boundary half-guard for \(P_1\) and \(P_2\) by Lemma~\ref{BigToSmall}, and thus \(\{{\tt g}\}\) monitors \(P = P_1 \cup P_2\) by Lemmas~\ref{Tri} and \ref{MiniQuad}.
\end{proof}

\section{Merging half-guard sets}\label{mergesec}

In this section we develop tools for constructing a CHG set for a polygon \(P\) by merging CHG sets for polygons which partition \(P\).

\begin{Lemma}\label{AddTri}
Assume \(P\) is a polygon monitored by an entire boundary half-guard \({\tt g}\) at \(b \in \partial P\). Let \(\overline{vw}\) be a side in \(P\) not containing \(b\). Then \(P\) contains the triangle \(T=[b,v,w]\) and \({\tt g}\) is an entire boundary half-guard in \(T\) which monitors \(T\).
\end{Lemma}
\begin{proof}
As \({\tt g}\) monitors \(P\), we have that each of the line segments \(\overline{bv}, \overline{bw}, \overline{vw}\) belong to \(P\). Since \(P\) is simple, we have \(T \subseteq P\), and therefore \({\tt g}\) is an entire boundary half-guard in \(T\) by Lemma~\ref{BigToSmall} and monitors \(T\) by Lemma~\ref{Tri}.
\end{proof}

\begin{Lemma}\label{JoinDigraphs}
Let \(V\) be a graph, and assume that \(V_1, V_2\) are connected subgraphs of \(V\), and every node of \(V\) belongs to at least one of \(V_1, V_2\).  Assume there exists \(x \in V_1\), \(y\in V_2\) which are adjacent in \(V\). Then \(V\) is connected.
\end{Lemma}
\begin{proof}
Let \(v,w\) be any nodes in \(V\). If \(v,w\) both belong to \(V_i\), for \(i \in \{1,2\}\), then since \(V_i\) is connected, there exists a path \(v\to w\) in \(V\). Assume then that \(v \in V_1\), \(w \in V_2\). Since \(V_1\) is connected, there exists a path \(v \to x\) in \(V_1\), and since \(V_2\) is connected, there exists a path \(y \to w\) in \(V_2\). Therefore the concatenation \(v \to x \to y \to w\) is a path in \(V\). Similarly, if \(v \in V_2\), \(w \in V_1\), then we have a  path \(v \to y \to x \to w\) in \(V\). Thus \(V\) is connected.
\end{proof}

\begin{Lemma}\label{IncludeGuard}
If \({\tt G}\) is a CHG set for \(P\), and \({\tt g}\) is an entire boundary half-guard in \(P\), then \({\tt G} \cup \{{\tt g}\}\) is a CHG set for \(P\).
\end{Lemma}
\begin{proof}
Since \({\tt G}\) monitors \(P\), there exists some \({\tt h} \in {\tt G}\) that sees \(x_{\tt g}\). Let \(u,v \neq x_{\tt g}\) be vertices adjacent to \(x_{\tt g}\) in \(P\), and let \(C\) be the cone formed by the interior angle \(\angle u x_{\tt g} v\). As \({\tt g}\) is an entire boundary half-guard, we have that \(C \subseteq H_{{\tt g}}\). As \({\tt h}\) sees \(x_{\tt g}\), the line segment \(\overline{x_{\tt h}x_{\tt g}}\) must be contained in \(P\), and therefore must also be contained in \(C\). Therefore \(\overline{x_{\tt h}x_{\tt g}}\) is contained in \(H_{{\tt g}} \cap P\), so \({\tt g}\) sees \({\tt h}\). Thus \(V_{{ \tt G} \cup \{\tt g\}}\) is connected by Lemma~\ref{JoinDigraphs}, completing the proof.
\end{proof}

\begin{Lemma}\label{SeeEachOther}
Let \(P\) be a polygon, and \(P_1, P_2\) be polygons which partition \(P\) and meet in a line segment \(s = P_1 \cap P_2\). For \(i \in \{1,2\}\), let \({\tt G}_i\) be an \(s\)-aligned CHG set for \(P_i\). 
Then \({\tt G} = {\tt G}_1 \cup {\tt G}_2\) is a CHG set for \(P\).
\end{Lemma}
\begin{proof}
It is clear that \({\tt G}\) monitors \(P\), so we must show only that \(V_{\tt G}\) is connected. For \(i \in \{1,2\}\), let \({\tt g}_i \in {\tt G}_i\) be an \(s\)-aligned half-guard in \(P_i\). 

First, assume that \(x_{{\tt g}_1} \in s\). Since \(x_{{\tt g}_2}\) sees all of \(s\), we have that \({\tt g}_2\) sees \({\tt g}_1\). Thus \(\overline{x_{{\tt g}_1}x_{{\tt g}_2}} \subset P\). But \(\overline{x_{{\tt g}_1}x_{{\tt g}_2}}\) is colinear to \(s\) as well, and is thus by assumption colinear to \(\partial H_{{\tt g}_1}\). Therefore \(\overline{x_{{\tt g}_1}x_{{\tt g}_2}} \subset H_{{\tt g}_1}\), and thus \({\tt g}_1\) sees \({\tt g}_2\). Thus it follows by Lemma~\ref{JoinDigraphs} that \(V_{\tt G}\) is connected. The case where \(x_{{\tt g}_2} \in s\) follows similarly.

Then assume that neither of \(x_{{\tt g}_1}\), \(x_{{\tt g}_2}\) lie on \(s\). Since \(s = P_1 \cap P_2\), we have that \(x_{{\tt g}_1} \notin P_2\) and \(x_{{\tt g}_2} \notin P_1\). Writing \(s = \overline{uv}\), we have then that \(x_{{\tt g}_1}, x_{{\tt g}_2}, u, v\) are colinear and distinct. Assume without loss of generality that \(u\) is the vertex of \(s\) closest to \(x_{{\tt g}_1}\). Then \(x_{{\tt g}_2}\) cannot lie on the line segment \(\overline{x_{{\tt g}_1}u}\), since by assumption \(\overline{x_{{\tt g}_1}u} \subset P_1\) because \({\tt g}_1\) sees \(u\) in \(P_1\). Similarly \(x_{{\tt g}_1}\) cannot lie on the line segment \(\overline{x_{{\tt g}_2}u}\) because \({\tt g}_2\) sees \(u\) in \(P_2\). Therefore \(x_{{\tt g}_2}\) must lie on the opposite side of \(s\) from \(x_{{\tt g}_1}\), so that \(v\) is the vertex of \(s\) closest to \(x_{{\tt g}_2}\). Therefore the ray \(\vv{x_{{\tt g}_1}u}\) contains the point \(x_{{\tt g}_2}\), and the ray \(\vv{x_{{\tt g}_2}v}\) contains the point \(x_{{\tt g}_1}\). Since \({\tt g}_1\) sees \(u\), it follows that \(x_{{\tt g}_2} \in H_{{\tt g}_1}\), and since \({\tt g}_2\) sees \(v\), it follows that \(x_{{\tt g}_1} \in H_{{\tt g}_2}\). Since \(\overline{x_{{\tt g}_1} x_{{\tt g}_2}} = \overline{x_{{\tt g}_1}u} \cup  \overline{uv} \cup \overline{v x_{{\tt g}_2}} \subseteq P_1 \cup P_2 = P\), it follows that \({\tt g}_1\) and \({\tt g}_2\) see each other in \(P\). Thus by Lemma~\ref{JoinDigraphs}, \(V_{{\tt G}}\) is connected.
\end{proof}

\section{Cooperative half-guard sets for arbitrary polygons}\label{mainsec}

In this section we prove our main theorem, which affirms that a CHG set of size \(\lfloor n/2\rfloor -1\) exists for every polygon with \(n\geq 4\) sides.

\begin{Theorem}\label{MainThm}
Let \(P\) be a polygon with \(n \geq 4\) sides. 
\begin{enumerate}
\item If \(n\) is odd, then there exists a CHG set \({\tt G}\) for \(P\) such that \(|{\tt G}| = (n-3)/2\).
\item If \(n\) is even, then for any side \(s\) in \(P\), there exists an \(s\)-aligned CHG set \({\tt G}\) for \(P\) such that \(|{\tt G}| =(n-2)/2\).
\end{enumerate}
\end{Theorem}
\begin{proof}
We go by induction on \(n\). The base case \(n = 4\) is the content of Lemma~\ref{Quad}. Now assume that \(P\) is a polygon with \(n>4\) sides, and the theorem statement holds for every polygon with fewer than \(n\) sides. We consider the even and odd cases for \(n\) separately.

{\em Assume \(n\) is even.} Let \(s = \overline{uv}\) be a side in \(P\). We proceed in cases (E1--E4) below.

(E1) Assume \(u,v\) are convex vertices in \(P\). By Lemma~\ref{OddOddDiag}, we may assume without loss of generality that there is another vertex \(w \in P\) such that \(\overline{uw}\) is a diagonal which decomposes \(P\) into polygons \(P_1, P_2\), with sides numbering \(n_1, n_2\), where \(n_1 + n_2 = n+2\) and \(n_1, n_2\) are odd.
\begin{align*}
\begin{array}{c}
\begin{tikzpicture}[scale=0.8]
\draw[very thick, black, line cap=round, fill=lightgray!10] (0,0)--(-1,1.5)--(1,3)--(1.5,2)--(2.3,3)--(5.7,3)--(5.35,2.4)--(5.9,1.5)--(6.5,2.7)--(7.5,1.8)--(8,3)--(11,1)--(11.5,1.5)--(11.5,2.2)--(10.5,1.7)-- (9,3)--(12,2.7)--(12.5,1)--(13.5,0.5)--(12,0.5)--(11,0)--(10,0)--(9,1.8)--(8,0)--(4.5,0.5)--(5,1.3)--(3.5,2.2)--(4,1.5)--(2.5,0)--(0,1)--(0,0);
\draw[very thick, black, line cap=round, fill=lightgray!50] (7.5,1.8)--(8,3)--(11,1)--(11.5,1.5)--(11.5,2.2)--(10.5,1.7)-- (9,3)--(12,2.7)--(12.5,1)--(13.5,0.5)--(12,0.5)--(11,0)--(10,0)--(9,1.8)--(8,0);
%
\draw[line width=3pt, black, line cap=round] (8,0)--(4.5,0.5);
%
%
%
\draw[very thick, cyan, dotted, line cap=round] (7.5,1.8)--(8,0);
%
\pgfmathsetmacro{\ex}{8}
\pgfmathsetmacro{\ey}{0}
\fill[fill = red, opacity = 0.4] (\ex,\ey) ++(-8:.5) arc (-8:172:.5);
\draw[thick, fill=red]  (8,0) circle (0.15);
%
%
%
\node[above] at (7.4,2.1) {$w$};
\node[below] at (8,-0.2) {$u$};
\node[left] at (7.7,0.5) {${\tt g}$};
\node[below] at (4.5,0.4) {$v$};
\node[] at (2.5,1.5) {$P_1$};
\node[] at (10.5,0.65) {$P_2$};
\end{tikzpicture}
\\
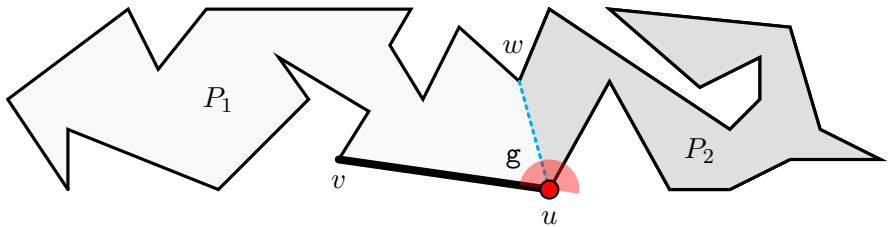
\scriptstyle{\textup{{ Fig.\,26: (E1) -- Convex endpoints on \(s\); diagonal decomposes \(P\) into odd-sided polygons}}}
\end{array}
\end{align*}

Place an entire boundary half-guard \({\tt g}\) for \(P\) at \(u\), with \(\partial H_{{\tt g}}\) colinear to \(s\), as in Fig.\,26. We have then that \({\tt g}\) is \(s\)-aligned. We further have that \({\tt g}\) is an entire boundary half-guard in both of \(P_1\) and \(P_2\). If either of these polygons is a triangle, then they are monitored by \({\tt g}\) by Lemma~\ref{Tri}. Let \(i \in \{1,2\}\). If \(n_i = 3\), then set \({\tt G}_i = \varnothing\). If \(n_i \geq 5\), then by the induction assumption we may let \({\tt G}_i\) be a CHG set for \(P_i\) such that \(|{\tt G}_i| = (n_i -3)/2\). Then we have by Lemma~\ref{IncludeGuard} that \({\tt G}_i \cup \{{\tt g}\}\) is a CHG set for \(P_i\). 
Then we have that \(P\) is monitored by the set \({\tt G} = {\tt G}_1 \cup { \{ \tt g}\} \cup {\tt G}_2\) of size
\begin{align*}
|{\tt G}|  = 
\frac{n_1 -3}{2} + 1 + \frac{n_2 -3}{2} = \frac{n_1 + n_2 - 4}{2}  = \frac{n-2}{2},
\end{align*}
and \(V_{\tt G}\) is a connected graph by Lemma~\ref{JoinDigraphs}. Therefore \({\tt G}\) is a CHG set for \(P\) which satisfies claim (ii). 

In cases (E2--E4) we then assume, without loss of generality, that the endpoint \(v\) of \(s = \overline{uv}\) is a reflex vertex in \(P\). Then the ray \(\vv{uv}\) passes into the interior of \(P\) and first intersects \(\partial P\) at some point \(b\). We may decompose \(P\) along the \(\vv{uv}\)-cut \(\overline{ub}\) into two polygons \(P_1\) and \(P_2\), of of size \(n_1\) and \(n_2\) respectively. We have that \(n_1, n_2\) are either both odd (E2), both even (E3), or opposite parity (E4), and we approach these cases separately below.

(E2) Assume \(n_1, n_2\) are odd. Then it follows that \(b\) is on the relative interior of a side of \(P\), and \(n_1 + n_2 = n+2\). Assume moreover that \(n_1, n_2\) are odd. Then we place an entire boundary half-guard \({\tt g}\) for \(P\) at \(b\), as in Fig.\,27. Then we have that \({\tt g}\) is \(s\)-aligned. With CHG sets \({\tt G}_1\), \({\tt G}_2\) for \(P_1\), \(P_2\) selected via the induction assumption, we may construct an \(s\)-aligned CHG set \({\tt G} = {\tt G}_1 \cup \{ {\tt g}\} \cup {\tt G}_2\) for \(P\) of cardinality \((n-2)/2\), with the argument proceeding exactly as in the case (E1).
\begin{align*}
\begin{array}{c}
\begin{tikzpicture}[scale=0.8]
\draw[very thick, black, line cap=round, fill=lightgray!10] (0,0)--(-1,1.5)--(1,3)--(1.5,2)--(2.3,3)--(5.7,3)--(5.35,2.4)--(5.9,1.5)--(6.5,2.7)--(7.5,1.8)--(8,3)--(11,1)--(11.5,1.5)--(11.5,2.2)--(10.5,1.7)-- (9,3)--(12,2.7)--(12.5,1)--(13.5,0.5)--(12,0.5)--(11,0)--(10,0)--(9,1.8)--(8,0)--(4.5,0.5)--(5,1.3)--(3.5,2.2)--(4,1.5)--(2.5,0)--(0,1)--(0,0);
\draw[very thick, black, line cap=round, fill=lightgray!50] (8,3)--(11,1)--(11.5,1.5)--(11.5,2.2)--(10.5,1.7)-- (9,3)--(12,2.7)--(12.5,1)--(13.5,0.5)--(12,0.5)--(11,0)--(10,0)--(9,1.8)--(8,0)--(6.9,0.15);
%
%
%
%
\draw[very thick, cyan, dotted, line cap=round] (6.9,0.2)--(8,3);
\draw[line width=3pt, black, line cap=round] (7.5,1.8)--(8,3);
%
%
%
%
\pgfmathsetmacro{\ex}{6.9}
\pgfmathsetmacro{\ey}{0.15}
\fill[fill = red, opacity = 0.4] (\ex,\ey) ++(-8:.5) arc (-8:172:.5);
\draw[thick, fill=red]  (6.9,0.15) circle (0.15);
\node[above] at (7.4,2.0) {$v$};
\node[above] at (8,3.1) {$u$};
\node[below] at (6.9,0.05) {${\tt g}$};
\node[] at (2.5,1.5) {$P_1$};
\node[] at (10.5,0.65) {$P_2$};
\end{tikzpicture}
\\
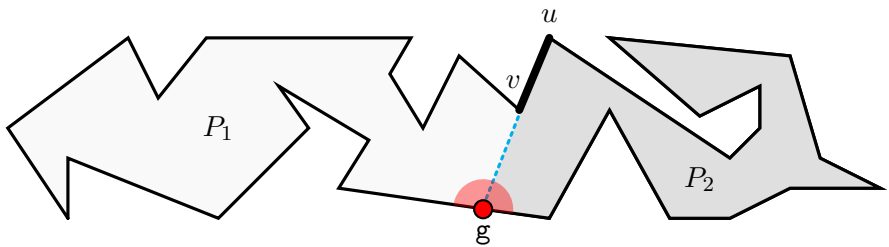
\scriptstyle{\textup{{ Fig.\,27: (E2) -- The \(\vv{uv}\)-cut partitions \(P\) into odd-sided polygons}}}
\end{array}
\end{align*}

(E3)  Assume \(n_1, n_2\) are even. Then it follows that \(b\) is on the relative interior of a side of \(P\), and \(n_1 + n_2 = n+2\). We may assume that we have labeled the polygons \(P_1\), \(P_2\) such that \(\overline{ub}\) is a side in \(P_1\) and \(\overline{vb}\) is a side in \(P_2\). By the induction assumption, there exists an \(\overline{ub}\)-aligned CHG set \({\tt G}_1\) for \(P_1\) such that \(|{\tt G}_1| = (n_1 - 2)/2\). Let \({\tt g}_1\) be the \(\overline{ub}\)-aligned half-guard in \({\tt G}_1\). Note that since \({\tt g}_1\) is \(\overline{ub}\)-aligned, we have that \({\tt g}_1\) is \(\overline{uv}\)- and \(\overline{vb}\)-aligned as well. There also exists by induction a \(\overline{vb}\)-aligned CHG set \({\tt G}_2\) for \(P_2\) such that \(|{\tt G}_2| = (n_2 - 2)/2\). Let \({\tt g}_2\) be the \(\overline{vb}\)-aligned half-guard in \({\tt G}_2\), as in Fig.\,28.
\begin{align*}
\begin{array}{c}
\begin{tikzpicture}[scale=0.8]
\draw[very thick, black, line cap=round, fill=lightgray!10] (0,0)--(-1,1.5)--(1,3)--(1.5,2)--(2.3,3)--(5.7,3)--(5.35,2.4)--(5.9,1.5)--(6.5,2.7)--(7.5,1.8)--(8,3)--(11,1)--(11.5,1.5)--(11.5,2.2)--(10.5,1.7)-- (9,3)--(12,2.7)--(12.5,1)--(13.5,0.5)--(12,0.5)--(11,0)--(10,0)--(9,1.8)--(8,0)--(4.5,0.5)--(5,1.3)--(3.5,2.2)--(4,1.5)--(2.5,0)--(0,1)--(0,0);
\draw[very thick, black, line cap=round, fill=lightgray!50] (5.9,1.5)--(6.5,2.7)--(7.5,1.8)--(8,3)--(11,1)--(11.5,1.5)--(11.5,2.2)--(10.5,1.7)-- (9,3)--(12,2.7)--(12.5,1)--(13.5,0.5)--(12,0.5)--(11,0)--(10,0)--(9,1.8)--(8,0)--(6.7,0.2);
%
%
%
\draw[line width=3pt, black, line cap=round] (5.35,2.4)--(5.9,1.5);
%
\draw[very thick, cyan, dotted, line cap=round] (5,3)--(6.7,0.2);
\draw[line width=3pt, black, line cap=round] (5.35,2.4)--(5.9,1.5);
%
%
%
%
\pgfmathsetmacro{\ex}{6.35}
\pgfmathsetmacro{\ey}{0.8}
\fill[fill = red, opacity = 0.4] (\ex,\ey) ++(-57:.5) arc (-57:123:.5);
\draw[thick, fill=red]  (6.35,0.8) circle (0.15);
\pgfmathsetmacro{\ex}{5}
\pgfmathsetmacro{\ey}{3}
\fill[fill = red, opacity = 0.4] (\ex,\ey) ++(180:.5) arc (180:360:.5);
\draw[thick, fill=red] (5,3) circle (0.15);
\node[below] at (6.7,0.1) {$b$};
\node[right] at (5.4,2.4) {$u$};
\node[right] at (5.95,1.55) {$v$};
\node[] at (2.5,1.5) {$P_1$};
\node[] at (10.5,0.65) {$P_2$};
\node[] at (4.5,2.5) {${\tt g}_1$};
\node[right] at (6.8,0.8) {${\tt g}_2$};
\end{tikzpicture}
\\
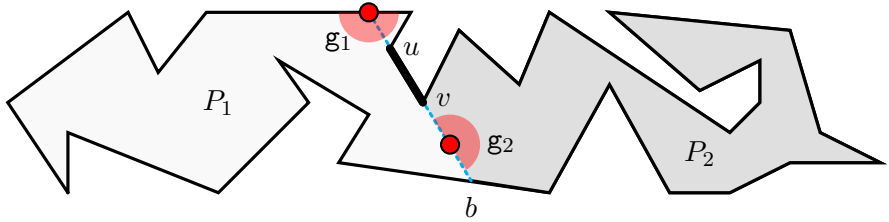
\scriptstyle{\textup{{ Fig.\,28: (E3) -- The \(\vv{uv}\)-cut partitions \(P\) into even-sided polygons}}}
\end{array}
\end{align*}
It follows then by Lemma~\ref{SeeEachOther} that \({\tt G} = {\tt G}_1 \cup {\tt G}_2\) is an \(s\)-aligned CHG set for \(P\). Since
\begin{align*}
|{\tt G}| = \frac{n_1 -2}{2} + \frac{n_2 -2}{2} = \frac{n_1 + n_2 - 4}{2} = \frac{n-2}{2},
\end{align*}
we have that \({\tt G}\) satisfies claim (ii).

(E4) Now assume that \(n_1, n_2\) have opposite parity. Then it follows that \(b\) is a vertex of \(P\), and \(n_1 + n_2 = n+1\). We assume without loss of generality that \(n_1\) is odd and \(n_2\) is even. By the induction assumption, there exists a 
CHG set \({\tt G}_1\) for \(P_1\) such that \(|{\tt G}_1| =(n_1-3)/2 \),
and there exists a \(\overline{vb}\)-aligned 
CHG set \({\tt G}_2\) for \(P_2\) such that \(|{\tt G}_2| =(n_2-2)/2 \).
Let \({\tt g}_2\) be the \(\overline{vb}\)-aligned guard in \({\tt G}_2\), as in Fig.\,29.
\begin{align*}
\begin{array}{cc}
\begin{tikzpicture}[scale=0.8]
\draw[very thick, black, line cap=round, fill=lightgray!10] (0,0)--(-1,1.5)--(1,3)--(1.5,2)--(2.3,3)--(5.7,3)--(5.35,2.4)--(5.9,1.5)--(6.5,2.7)--(7.5,1.8)--(8,3)--(11,1)--(11.5,1.5)--(11.5,2.2)--(10.5,1.7)-- (9,3)--(12,2.7)--(12.5,1)--(13.5,0.5)--(12,0.5)--(11,0)--(10,0)--(9,1.8)--(8,0)--(4.5,0.5)--(5,1.3)--(3.5,2.2)--(4,1.5)--(2.5,0)--(0,1)--(0,0);
\draw[very thick, black, line cap=round, fill=lightgray!50] (1.5,2)--(2.3,3)--(5.7,3)--(5.35,2.4)--(5.9,1.5)--(6.5,2.7)--(7.5,1.8)--(8,3)--(11,1)--(11.5,1.5)--(11.5,2.2)--(10.5,1.7)-- (9,3)--(12,2.7)--(12.5,1)--(13.5,0.5)--(12,0.5)--(11,0)--(10,0)--(9,1.8)--(8,0)--(4.5,0.5)--(5,1.3)--(3.5,2.2)--(4,1.5)--(2.5,0);
%
%
%
%
\draw[very thick, cyan, dotted, line cap=round] (1,3)--(2.5,0);
\draw[line width=3pt, black, line cap=round] (1,3)--(1.5,2);
%
%
\pgfmathsetmacro{\ex}{2}
\pgfmathsetmacro{\ey}{1}
\fill[fill = red, opacity = 0.4] (\ex,\ey) ++(-65:.5) arc (-65:115:.5);
\draw[thick, fill=red]  (2,1) circle (0.15);
%
%
\pgfmathsetmacro{\ex}{1.5}
\pgfmathsetmacro{\ey}{2}
\fill[fill = red, opacity = 0.4] (\ex,\ey) ++(115:.5) arc (115:295:.5);
\draw[thick, fill=red]  (1.5,2) circle (0.15);
\node[below] at (2.5,-0.1) {$b$};
\node[above] at (1,3) {$u$};
\node[above] at (1.6,2.25) {$v$};
\node[] at (0.5,1.5) {$P_1$};
\node[] at (6.8,1.1) {$P_2$};
\node[right] at (1.6,2) {${\tt g}$};
\node[right] at (2.4,1) {${\tt g}_2$};
\end{tikzpicture}
\\
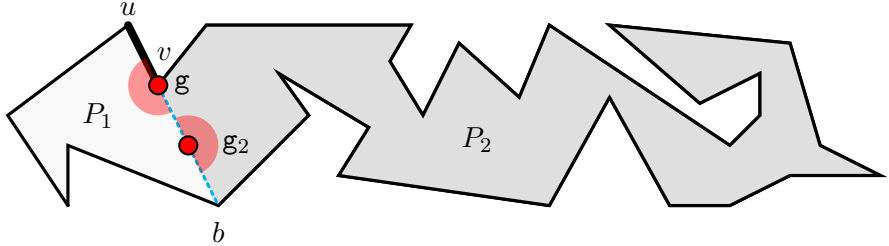
\scriptstyle{\textup{{ Fig.\,29: (E4) -- The \(\vv{uv}\)-cut partitions \(P\) into even- and odd-sided polygons}}}
\end{array}
\end{align*}
Let \({\tt g}\) be a half-guard such that \(x_{\tt g} = v\), with half-plane \(H_{\tt g}\) oriented such that \(\partial H_{\tt g}\) is colinear to \(\overline{ub}\) and \({\tt g}\) is an entire boundary half-guard for \(P_1\). Note then that \({\tt g}\) sees all of \(\overline{uv}\) and \(\overline{vb}\) as well, so that \({\tt g}\) is \(\overline{uv}\)-aligned and \(\overline{vb}\)-aligned. Then it follows by Lemma~\ref{IncludeGuard} that \(V_{{\tt G}_1 \cup \{{\tt g}\}}\) is  connected, and then by Lemma~\ref{SeeEachOther} that \({\tt G} = {\tt G}_1 \cup \{{\tt g}\} \cup {\tt G}_2\) is a CHG set for \(P\). Since \({\tt g}\) is \(s\)-aligned and
\begin{align*}
|{\tt G}| = \frac{n_1-3}{2} + 1 + \frac{n_2 -2}{2} = \frac{n_1 + n_2 - 3}{2} = \frac{n-2}{2},
\end{align*}
we have that \({\tt G}\) satisfies claim (ii).

This concludes the proof of the induction step when \(n\) is even.

{\em Assume \(n\) is odd}. By Lemma~\ref{PentDecomp} there exists \(0 \leq k\leq  5\) and a decomposition \(\mathcal{D} = \{P_0, P_1, \ldots, P_k\}\) of \(P\) such that \(P_0\) is a pentagon, \(P_1, \ldots, P_k\) are even-sided polygons, \(P_i\) and \(P_j\) have no sides in common for \(1 \leq i, j \leq k\), \(i \neq j\), and each shares exactly one side with \(P_0\). Letting \(n_i \) represent the number of sides of \(P_i\), we have
\(
n_1 + \cdots + n_k = n + 2k - 5.
\)

By Lemma~\ref{Pent},
there exists an entire boundary half-guard \({\tt g}\) for \(P_0\) such that
\(\{ {\tt g}\}\) monitors \(P_0\), and \(x_{\tt g}\) is not a vertex of \(P_0\), as in Fig.\,30.
 If \(k = 0\), then \(P = P_0\), so \(n=5\) and \(\{{\tt g}\}\) satisfies (i). Thus we assume \(k \geq 1\). Let \(i \in \{1, \ldots, k\}\). Assume that \(s = \overline{uv}\) is the edge shared by \(P_0\) and \(P_i\). We now construct a CHG set for \(P_i\) via the induction assumption.

 \begin{align*}
\begin{array}{cc}
\begin{tikzpicture}[scale=0.8]
\draw[very thick, black, line cap=round, fill=lightgray!10] (3,0)--(1,4)--(3,7)--(3,3)--(6,3)--(5,1)--(8,2)--(5,6)--(13,7)--(8.1,5)--(12.5,4.5)--(12,6)--(15,3)--(10.5,3)--(10,2)--(13,1)--(10,0)--(3,0);
\fill[ line cap=round, fill=lightgray!50!] (5,1)--(8,2)--(8,5)--(10,2)--(13,1)--(5,1);
\draw[very thick, black, line cap=round] (3,0)--(1,4)--(3,7)--(3,3)--(6,3)--(5,1)--(8,2)--(5,6)--(13,7)--(8.1,5)--(12.5,4.5)--(12,6)--(15,3)--(10.5,3)--(10,2)--(13,1)--(10,0)--(3,0);
%
%
%
%
\draw[very thick, cyan, dotted, line cap=round] (2.5,1)--(13,1);
\draw[very thick, cyan, dotted, line cap=round] (8,1)--(8,6.3);
\draw[very thick, cyan, dotted, line cap=round] (10,2)--(8,5);
\draw[very thick, cyan, dotted, line cap=round] (5,1)--(1.4,4.6);
\draw[very thick, cyan, dotted, line cap=round] (8,5)--(15,3);
\draw[very thick, cyan, dotted, line cap=round] (12.5,4.5)--(12.8,3.6);
\draw[very thick, cyan, dotted, line cap=round] (10.5,3)--(11.05,4.15);
\draw[very thick, cyan, dotted, line cap=round] (10,2)--(8,1);
%
\pgfmathsetmacro{\ex}{3.4}
\pgfmathsetmacro{\ey}{1}
\fill[fill = red, opacity = 0.4] (\ex,\ey) ++(180:.5) arc (180:360:.5);
\draw[thick, fill=red]  (3.4,1) circle (0.15);
\pgfmathsetmacro{\ex}{1.4}
\pgfmathsetmacro{\ey}{4.6}
\fill[fill = red, opacity = 0.4] (\ex,\ey) ++(237:.5) arc (237:417:.5);
\draw[thick, fill=red]  (1.4,4.6) circle (0.15);
\pgfmathsetmacro{\ex}{5}
\pgfmathsetmacro{\ey}{1}
\fill[fill = red, opacity = 0.4] (\ex,\ey) ++(0:.5) arc (0:180:.5);
\draw[thick, fill=red]  (5,1) circle (0.15);
\pgfmathsetmacro{\ex}{8}
\pgfmathsetmacro{\ey}{1}
\fill[fill = red, opacity = 0.4] (\ex,\ey) ++(0:.5) arc (0:180:.5);
\draw[thick, fill=red]  (8,1) circle (0.15);
\pgfmathsetmacro{\ex}{8}
\pgfmathsetmacro{\ey}{6.35}
\fill[fill = red, opacity = 0.4] (\ex,\ey) ++(187:.5) arc (187:367:.5);
\draw[thick, fill=red]  (8,6.35) circle (0.15);
\pgfmathsetmacro{\ex}{11.04}
\pgfmathsetmacro{\ey}{4.14}
\fill[fill = red, opacity = 0.4] (\ex,\ey) ++(164:.5) arc (164:344:.5);
\draw[thick, fill=red]  (11.04,4.14) circle (0.15);
\pgfmathsetmacro{\ex}{12.8}
\pgfmathsetmacro{\ey}{3.62}
\fill[fill = red, opacity = 0.4] (\ex,\ey) ++(-16:.5) arc (-16:164:.5);
\draw[thick, fill=red]  (12.8,3.62) circle (0.15);
%
%
%
\node[below] at (8.75,2.8) {$P_0$};
\node[above] at (9.7,3.4) {$P_1$};
\node[above] at (6.9,4.6) {$P_3$};
\node[] at (3,2) {$P_2$};
\node[below] at (8,0.8) {${\tt g}$};
\node[] at (0.9,4.8) {${\tt g}_1^2$};
\node[below] at (4,1) {${\tt g}_2^2$};
\node[below] at (5.3,1) {${\tt g}_3^2$};
\node[] at (8,7) {${\tt g}_1^3$};
\node[] at (11.4,3.5) {${\tt g}_1^1$};
\node[] at (13,4.4) {${\tt g}_2^1$};
\end{tikzpicture}
\\
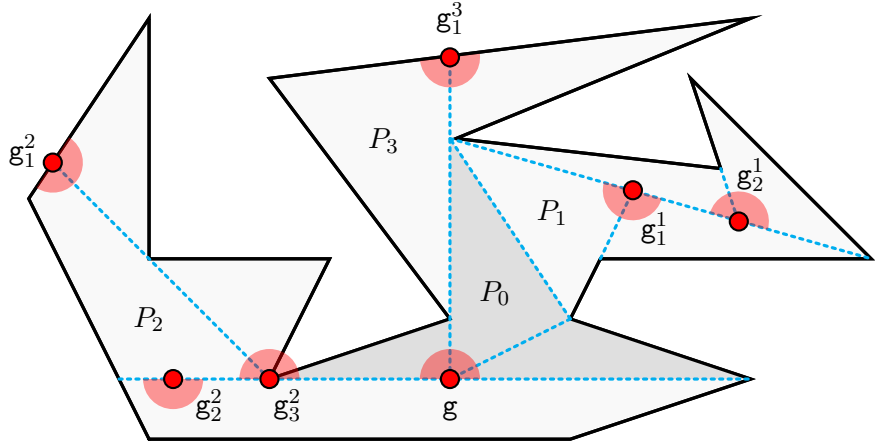
\scriptstyle{\textup{{ Fig.\,30: Decomposition of \(P\) into polygons \(P_0, \ldots, P_3\), with \({\tt g}\) and CHG sets \({\tt G}_i = \{g^i_1, \ldots, g^i_{(n_i -2)/2}\}\)}}}
\end{array}
\end{align*}

First assume that \(x_{\tt g}\) is not colinear to \(s\) (as is the case with \(P_1\) in Fig.\,30). Then by Lemma~\ref{AddTri}, \({\tt g}\) is an entire boundary half-guard in \(T = [x_{\tt g}, u,v] \subseteq P_0\). Then \(P_i' = P_i \cup T\) is a polygon with \(n_i + 1\) sides, and \({\tt g}\) is an entire boundary half-guard in \(P_i'\). By the induction assumption, there exists a CHG set \({\tt G}_i\) for \(P_i'\) of size \((n_i -2)/2\). By Lemma~\ref{IncludeGuard}, we have that \( V_{{\tt G}_i \cup \{{\tt g}\}}\) is  connected.

Next assume that \(x_{\tt g}\) is colinear to \(s\). If \(x_{\tt g} \in s\) (as is the case with \(P_2\) in Fig.\,30), then, since \(x_{\tt g}\) is in the relative interior of \(s\) and an entire boundary half-guard for \(s\), it follows that \(\partial H_{\tt g}\) is colinear to \(s\). If \(x_{\tt g} \notin s\) (as is the case with \(P_3\) in Fig.\,30), we nonetheless have that \({\tt g}\) sees all of \(s\) since \(\{{\tt g}\}\) monitors \(P_0\). Therefore, in either case, \({\tt g}\) is \(s\)-aligned in \(P_0\). By the induction assumption, there exists an \(s\)-aligned CHG set \({\tt G}_i\) for \(P_i\) of size \((n_i - 2)/2\). By Lemma~\ref{SeeEachOther}, it follows then that \({\tt G}_i \cup \{{\tt g}\}\) is a CHG set for \(P_i \cup P_0\), so \(V_{{\tt G}_i \cup \{{\tt g}\}}\) is  connected.

With half-guard sets \({\tt G}_i\), for \(i \in \{1, \ldots, k\}\) constructed as above, let 
\({\tt G} := {\tt G}_1 \cup \cdots \cup {\tt G}_k \cup \{{\tt g}\}\). Each \(V_{{\tt G}_i \cup \{{\tt g}\}}\) graph is connected, so it follows by 
inductive application of Lemma~\ref{JoinDigraphs} that \(V_{{\tt G}}\) is connected. As \({\tt G}_i\) monitors \(P_i\) for all \(i\), and \(\{{\tt g}\}\) monitors \(P_0\), we have that \({\tt G}\) monitors \(P\) as well. Finally, we note that
\begin{align*}
|{\tt G}| = 1 + \sum_{i=1}^k |{\tt G}_i| = 1 + \sum_{i=1}^k \frac{n_i -2}{2} = \frac{n_1 + \cdots + m_k - 2k + 2}{2} = \frac{n-3}{2},
\end{align*}
so \({\tt G}\) is a CHG set for \(P\) that satisfies condition (i). This completes the induction step, and the proof.
\end{proof}

\begin{Corollary}\label{MainCor}
For all \(n \geq 4\), we have
\(
\textup{\(\cg(n,180^\circ)\)} = \lfloor n/2\rfloor - 1
\).
\end{Corollary}
\begin{proof}
Noting that
\begin{align*}
\lfloor n/2 \rfloor -1 
=
\begin{cases}
(n-2)/2 & \textup{if \(n\) is even};\\
(n-3)/2 & \textup{if \(n\) is odd},\\
\end{cases}
\end{align*}
we have by Theorem~\ref{MainThm} that for every polygon with \(n\) sides, there exists a CHG set of cardinality \(\lfloor n/2 \rfloor -1\), so \(\cg(n,180^\circ) \leq \lfloor n/2\rfloor - 1\). We additionally have by Theorem~\ref{HPthm} that \(\cg(n,180^\circ) \geq \cg(n,360^\circ) =  \lfloor n/2\rfloor - 1\), so the claim follows.
\end{proof}

\section{Cooperative half-guard sets for orthogonal polygons}\label{mainsecOrth}

In this section we show that the result of Theorem~\ref{MainThm} may be improved upon slightly when \(P\) is an orthogonal polygon.

\begin{Theorem}\label{MainThmOrth}
Let \(P\) be an orthogonal polygon with \(n \geq 6\) sides. Then there exists a CHG set \({\tt G}\) for \(P\) such that \(|{\tt G}| =n/2 -2.\) 
\end{Theorem}
\begin{proof}
We go by induction on \(n\). Let \(P\) be an orthogonal polygon with \(n \geq 6\) sides. Then \(P\) has a reflex vertex \(v\). Let \(u\) be a vertex adjacent to \(v\). The ray \(\vv{uv}\) passes into the interior of \(P\) and first intersects \(\partial P\) at some point \(b\). Then we may partition \(P\) along the \(\vv{uv}\)-cut \(\overline{vb}\) into two orthogonal polygons \(P_1\), \(P_2\), with sides numbering \(n_1, n_2\) respectively. 

First, note that if \(n = 6\), then \(n_1 = n_2 = 4\), and \(b\) must be on the relative interior of an edge of \(P\). Place an entire boundary half-guard \({\tt g}\) for \(P\) at \(b\). Then \(\{{\tt g}\}\) monitors \(P = P_1 \cup P_2\) by Lemma~\ref{Tri}. This serves as the base case for our induction argument, so we now assume that \(n >6\) and that the claim holds for \(n' < n\). The point \(b\) described above must be on the relative interior of an edge of \(P\) or at a reflex vertex of \(P\). We consider these cases separately.

Assume \(b\) is on the relative interior of an edge of \(P\). Then \(n_1 + n_2 = n+2\). Place an entire boundary half-guard \({\tt g}\) for \(P\) at \(b\). Then \({\tt g}\) is an entire boundary half-guard in both \(P_1\) and \(P_2\). If \(n_i = 4\) for \(i \in \{1,2\}\), then \(P_i\) is a rectangle, and thus \(\{ \tt g\}\) monitors \(P_i\) by Lemma~\ref{Tri}. Now assume that \(n > 6\). Let \(i \in \{1,2\}\). If \(n_i = 4\), set \({\tt G}_i = \varnothing\). Otherwise by the induction assumption, let \({\tt G}_i\) be a CHG set for \(P_i\) of size \(n_i/2-2\). Note that by Lemma~\ref{IncludeGuard}, \({ \tt G}_i \cup \{\tt g\}\) is a CHG set for \(P_i\). Letting \({\tt G} = {\tt G}_1 \cup \{ \tt g\} \cup {\tt G}_2\) we have by Lemma~\ref{JoinDigraphs} that \(V_{\tt G}\) is connected and \({\tt G}\) monitors \(P = P_1 \cup P_2\), so \({\tt G}\) is a CHG set for \(P\) of size
\begin{align*}
|{\tt G}| = \left(\frac{n_1}{2}-2\right) + 1 + \left(\frac{n_2}{2}-2\right) = \frac{n_1 + n_2 -2}{2} -2= \frac{n}{2}-2,
\end{align*}
as required.

Assume \(b\) is a reflex vertex in \(P\). Then \(n_1 + n_2 = n\), so we have \(n >6\). Let \(i \in \{1,2\}\), and place an entire boundary half-guard \({\tt g}_i\) for \(P_i\) on the relative interior of the edge \(\overline{vb}\). As above, if \(n_i = 4\), then \(\{{\tt g}\}\) monitors \(P_i\), so if \(n_i = 4\), set \({\tt G}_i = \varnothing\). Otherwise by the induction assumption, let \({\tt G}_i\) be a CHG set for \(P_i\) of size \(n_i/2-2\). Note that by Lemma~\ref{IncludeGuard}, \({ \tt G}_i \cup \{{\tt g}_i\}\) is a \(\overline{vb}\)-aligned CHG set for \(P_i\). It follows then by Lemma~\ref{SeeEachOther} that \({\tt G} = {\tt G}_1 \cup \{ {\tt g}_1\} \cup \{ {\tt g}_2\} \cup {\tt G}_2\) is a CHG set for \(P = P_1 \cup P_2\) of size
\begin{align*}
|{\tt G}| =\left( \frac{n_1}{2}-2 \right)+ 2 + \left( \frac{n_2}{2}-2 \right) = \frac{n_1 + n_2}{2} -2= \frac{n}{2}-2,
\end{align*}
as required. This completes the induction step, and the proof.
\end{proof}

 \begin{align*}
\begin{array}{cc}
\begin{tikzpicture}[scale=0.8]
\draw[very thick, black, line cap=round, fill=lightgray!10] (1,0)--(1,1)--(3,1)--(3,4)--(2,4)--(2,2)--(0,2)--(0,5)--(5,5)--(5,6)--(12,6)--(12,5)--(9,5)--(9,4)--(15,4)--(15,1)--(12,1)--(12,0)--(9,0)--(9,1.75)--(10,1.75)--(10,3)--(6,3)--(6,1)--(8,1)--(8,0)--(1,0);
%
%
%
%
\draw[very thick, cyan, dotted, line cap=round] (3,1)--(6,1);
\draw[very thick, cyan, dotted, line cap=round] (3,3)--(6,3);
\draw[very thick, cyan, dotted, line cap=round] (0,4)--(2,4);
\draw[very thick, cyan, dotted, line cap=round] (5,5)--(9,5);
\draw[very thick, cyan, dotted, line cap=round] (3,4)--(3,5);
\draw[very thick, cyan, dotted, line cap=round] (9,3)--(9,4);
\draw[very thick, cyan, dotted, line cap=round] (10,3)--(15,3);
\draw[very thick, cyan, dotted, line cap=round] (12,0)--(12,3);
\draw[very thick, cyan, dotted, line cap=round] (10,1.75)--(12,1.75);
%
%
\pgfmathsetmacro{\ex}{0}
\pgfmathsetmacro{\ey}{4}
\fill[fill = red, opacity = 0.4] (\ex,\ey) ++(-90:.5) arc (-90:90:.5);
\draw[thick, fill=red]  (\ex,\ey) circle (0.15);
\pgfmathsetmacro{\ex}{3}
\pgfmathsetmacro{\ey}{3}
\fill[fill = red, opacity = 0.4] (\ex,\ey) ++(-90:.5) arc (-90:90:.5);
\draw[thick, fill=red]  (\ex,\ey) circle (0.15);
\pgfmathsetmacro{\ex}{5.25}
\pgfmathsetmacro{\ey}{1}
\fill[fill = red, opacity = 0.4] (\ex,\ey) ++(0:.5) arc (0:180:.5);
\draw[thick, fill=red]  (\ex,\ey) circle (0.15);
\pgfmathsetmacro{\ex}{6}
\pgfmathsetmacro{\ey}{5}
\fill[fill = red, opacity = 0.4] (\ex,\ey) ++(0:.5) arc (0:180:.5);
\draw[thick, fill=red]  (\ex,\ey) circle (0.15);
\pgfmathsetmacro{\ex}{9}
\pgfmathsetmacro{\ey}{3}
\fill[fill = red, opacity = 0.4] (\ex,\ey) ++(0:.5) arc (0:180:.5);
\draw[thick, fill=red]  (\ex,\ey) circle (0.15);
\pgfmathsetmacro{\ex}{15}
\pgfmathsetmacro{\ey}{3}
\fill[fill = red, opacity = 0.4] (\ex,\ey) ++(90:.5) arc (90:270:.5);
\draw[thick, fill=red]  (\ex,\ey) circle (0.15);
\pgfmathsetmacro{\ex}{12}
\pgfmathsetmacro{\ey}{1.75}
\fill[fill = red, opacity = 0.4] (\ex,\ey) ++(90:.5) arc (90:270:.5);
\draw[thick, fill=red]  (\ex,\ey) circle (0.15);
\pgfmathsetmacro{\ex}{3.75}
\pgfmathsetmacro{\ey}{1}
\fill[fill = red, opacity = 0.4] (\ex,\ey) ++(180:.5) arc (180:360:.5);
\draw[thick, fill=red]  (\ex,\ey) circle (0.15);
\pgfmathsetmacro{\ex}{8}
\pgfmathsetmacro{\ey}{5}
\fill[fill = red, opacity = 0.4] (\ex,\ey) ++(180:.5) arc (180:360:.5);
\draw[thick, fill=red]  (\ex,\ey) circle (0.15);
\pgfmathsetmacro{\ex}{3}
\pgfmathsetmacro{\ey}{5}
\fill[fill = red, opacity = 0.4] (\ex,\ey) ++(180:.5) arc (180:360:.5);
\draw[thick, fill=red]  (\ex,\ey) circle (0.15);
\pgfmathsetmacro{\ex}{12}
\pgfmathsetmacro{\ey}{3}
\fill[fill = red, opacity = 0.4] (\ex,\ey) ++(180:.5) arc (180:360:.5);
\draw[thick, fill=red]  (\ex,\ey) circle (0.15);
\end{tikzpicture}
\\
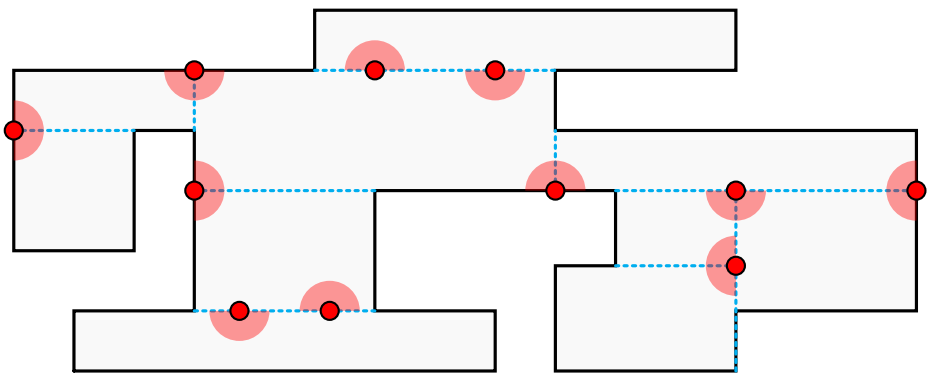
\scriptstyle{\textup{{ Fig.\,31: Orthogonal polygon with \(26\) sides and CHG set of cardinality \(11\)}}}
\end{array}
\end{align*}

\begin{Corollary}\label{MainCorOrth}
For all \(n \geq 6\), we have
\(
\textup{\(\cg^\perp(n,180^\circ)\)} =n/2-2.
\)
\end{Corollary}
\begin{proof}
We have that \(\cg^\perp(n,180^\circ) \geq \cg^\perp(n,360^\circ) =n/2-2\) by Theorem~\ref{HPthm}, and we have by Theorem~\ref{MainThmOrth} that \(\cg^\perp(n,180^\circ) \leq n/2-2\), so the claim follows.
\end{proof}

\color{black}

\end{document}